\documentclass[a4paper,notitlepage,twoside,reqno,11pt]{amsart}

\usepackage{anysize}
\marginsize{3.4cm}{3.4cm}{3cm}{3cm}

\usepackage{bbm,pifont,latexsym,dcolumn,indentfirst}
\usepackage[hypertexnames=false,bookmarksopen=true,linktocpage=true,pdfstartview={XYZ null null 1.25}]{hyperref}
\usepackage{amsthm,amsmath,amssymb,amscd,amsfonts,mathrsfs}
\usepackage{color,graphicx,xcolor,graphics,subfigure,extarrows,caption2}
\usepackage{pdfpages,titletoc}

\usepackage{autonum}

\newtheorem{thm}{Theorem}[section]
\newtheorem{lem}[thm]{Lemma}
\newtheorem{cor}[thm]{Corollary}
\newtheorem{prop}[thm]{Proposition}

\theoremstyle{definition}
\newtheorem*{defi}{Definition}
\newtheorem*{rmk}{Remark}

\newcommand{\EC}{\widehat{\mathbb{C}}}

\newcommand{\C}{\mathbb{C}}
\newcommand{\D}{\mathbb{D}}

\newcommand{\N}{\mathbb{N}}
\newcommand{\Q}{\mathbb{Q}}
\newcommand{\R}{\mathbb{R}}
\newcommand{\T}{\mathbb{T}}

\newcommand{\MB}{\mathcal{B}}
\newcommand{\MC}{\mathcal{C}}

\newcommand{\MO}{\mathcal{O}}
\newcommand{\MP}{\mathcal{P}}

\newcommand{\MV}{\mathcal{V}}

\newcommand{\MCV}{\mathcal{CV}}

\newcommand{\ii}{\textup{i}}

\newcommand{\diam}{\textup{diam}}
\newcommand{\dist}{\textup{dist}}

\newcommand{\Int}{\textup{int}}

\makeatletter\@addtoreset{equation}{section}\makeatother

\begin{document}

\author[F. Yang]{FEI YANG}
\address{School of Mathematics, Nanjing University, Nanjing 210093, P. R. China}
\email{yangfei@nju.edu.cn}

\author[G. Zhang]{GAOFEI ZHANG}
\address{School of Mathematics, Nanjing University, Nanjing 210093, P. R. China}
\email{zhanggf@nju.edu.cn}

\author[Y. Zhang]{YANHUA ZHANG}
\address{School of Mathematical Sciences, Qufu Normal University, Qufu 273165, P. R. China}
\email{zhangyh0714@qfnu.edu.cn}

\title[Local connectivity of transcendental Julia sets]{Local connectivity of Julia sets of some transcendental entire functions with Siegel disks}

\begin{abstract}
Based on the weak expansion property of a long iteration of a family of quasi-Blaschke products near the unit circle established recently, we prove that the Julia sets of a number of transcendental entire functions with bounded type Siegel disks are locally connected.
In particular, if $\theta$ is of bounded type, then the Julia set of the sine function $S_\theta(z)=e^{2\pi\ii\theta}\sin(z)$ is locally connected.
Moreover, we prove the existence of transcendental entire functions having Siegel disks and locally connected Julia sets with asymptotic values.
\end{abstract}



\date{\today}



\maketitle


\section{Introduction}

\subsection{Backgrounds}

Let $f$ be a non-linear holomorphic map with an irrationally indifferent fixed point $z_0$, i.e., $f(z_0)=z_0$ and $f'(z_0)=e^{2\pi\ii\theta}$ for some $\theta\in\R\setminus\Q$. If $f$ is locally conjugate to the linear map $R_\theta(\zeta)=e^{2\pi\ii\theta}\zeta$ in a neighborhood of $z_0$, then the maximal linearization domain containing $z_0$ is called the \textit{Siegel disk} of $f$ centered at $z_0$.
An irrational number $\theta$ is said to be of \textit{bounded type}, if the continued fraction expansion $\theta=[a_0;a_1,a_2,\cdots, a_n,\cdots]$ satisfies $\sup_{n\geqslant 1}\{a_n\}<+\infty$.
It was well known that if $\theta$ is of bounded type (actually, more general, if $\theta$ is a Diophantine number), then $f$ is locally linearizable at $z_0$ and $f$ has a Siegel disk containing $z_0$ \cite{Sie42}.

The topological properties of the boundaries of the Siegel disks of holomorphic maps have been studied extensively in the past four decades. They are motivated by Douady-Sullivan's conjecture in the 1980s which predicts that the Siegel disk boundaries of rational maps with degree at least two are Jordan curves. For rational maps, this conjecture has been settled under the assumption that the rotation number $\theta$ is of bounded type \cite{Zha11} (see \cite{Dou87}, \cite{Her87}, \cite{Zak99}, \cite{Shi01} for partial results). One may also refer to \cite{PZ04}, \cite{ABC04}, \cite{BC07}, \cite{Zha14}, \cite{She18a}, \cite{Che22b}, \cite{SY24} for related progresses. Besides rational maps, this conjecture has also been studied for some transcendental entire functions, such as the maps with the form $P(z)e^{Q(z)}$, where $P$ and $Q$ are polynomials \cite{Zak10} (see \cite{Gey01}, \cite{KZ09}, \cite{BF10}, \cite{Kat11}, \cite{KN22b} for partial results), the sine family and its generalizations (see \cite{Zha05}, \cite{Yan13}, \cite{Zha16}, \cite{ZFS20}), and some other families (see \cite{Che06}, \cite{CE18}). In particular, under the assumption that the rotation number is of bounded type, the Siegel disk boundaries mentioned above are proved to be quasi-circles containing at least one critical point.

For the topology of the whole Julia sets of rational maps containing Siegel disks, Petersen proved that the Julia sets of quadratic polynomials with a fixed bounded type Siegel disk are locally connected \cite{Pet96} (see \cite{Yam99} for an alternative proof). In these proofs, the Douady-Ghys surgery on analytic Blaschke products and the Petersen puzzles are essential. See \cite{PZ04}, \cite{She18a} and \cite{Yan23J} for further results based on these tools.
As an extension of Petersen's results (\cite{Pet96}, \cite{Pet98}), the local connectivity of the Julia sets of a class of rational maps with bounded type Siegel disks has been proved in \cite{WYZZ25} recently in which case neither the analytic Blaschke models nor puzzle structures are available.
For a rough survey of the local connectivity of Julia sets of other rational maps, see the introduction in \cite{WYZZ25}.

For transcendental entire functions, the local connectivity of Julia sets is also significant for understanding the global dynamics and this property was studied for a number of classes.
For examples, the cases of hyperbolic \cite{Mor99}, \cite{BFR15}, semi-hyperbolic \cite{BM02}, strongly geometrically finite \cite{ARS22} and strongly postcritically separated \cite{Par22} etc.
See also \cite{BD00b}, \cite{Mih12} and \cite{Osb13} for related results.
Recently, the local connectivity of the Julia sets of some transcendental meromorphic functions with tame Fatou components was studied in \cite{BFJK25}.
Our main aim in this paper is to study the local connectivity of the Julia sets of some transcendental entire functions containing bounded type Siegel disks.

\subsection{Main results}

Let $f$ be a transcendental entire function. We use $F(f)$ and $J(f)$ to denote the \textit{Fatou set} and \textit{Julia set} of $f$ respectively. Let $CV(f)$ and $AV(f)$ be the \textit{critical values} and \textit{asymptotic values} of $f$ in $\C$ respectively. The singular set $S(f)$ of $f$ is the collection of all \textit{singular values} of the inverse function $f^{-1}$, which is the closure of  $CV(f)\cup AV(f)$ in $\C$.
The \textit{postsingular set} of $f$ is
\begin{equation}\label{equ:P-f}
\MP(f):=\overline{\bigcup_{n\geqslant 0}f^{\circ n}\big(S(f)\big)}.
\end{equation}
Let $\MB:=\{f \text{ is transcendental entire}: S(f) \text{ is bounded}\}$ be the \textit{Eremenko-Lyubich} class (compare \cite{EL92}).
To obtain the local connectivity of the Julia set of $f$ containing Siegel disks, we require that the dynamics of $f$ on $\MP(f)$ is tame.

\begin{thm}\label{thm:entire}
Let $f\in \MB$ be a transcendental entire function with Siegel disks and no asymptotic values. Suppose that
\begin{enumerate}
\item every Siegel disk of $f$ is of bounded type with quasi-circle boundary in $\C$;
\item $CV(f)\cap J(f)$ is finite, $\MP(f)\cap J(f)$ is finite outside the union of all Siegel disk boundaries and $\MP(f)\cap F(f)$ is compact; and
\item the number of critical points in every component of $F(f)$ and the local degrees of all critical points of $f$ have a uniform upper bound.
\end{enumerate}
Then $J(f)$ is locally connected.
\end{thm}

Note that each function in Theorem \ref{thm:entire} contains at least one recurrent critical point in the boundary of every cycle of Siegel disks and hence in the Julia set (see \cite{GS03}). To the best of our knowledge, this is the first result of the local connectivity of the Julia sets of transcendental entire functions containing Siegel disks.

For the sine family $S_\theta(z)=e^{2\pi\ii\theta}\sin (z)$ with bounded type $\theta$, it was proved in \cite{Zha05} that $S_\theta$ contains a Siegel disk centered at the origin which is bounded by a quasi-circle. Moreover, all critical point of $S_\theta$ have local degree two and they are mapped to the two symmetric critical values which lie on the boundary of the Siegel disk. Note that $S_\theta$ has no finite asymptotic values. Therefore, as an immediate application of Theorem \ref{thm:entire}, we have the following result. See Figure \ref{Fig:Siegel-sin}.

\begin{figure}[htbp]
  \setlength{\unitlength}{1mm}
  \includegraphics[width=0.95\textwidth]{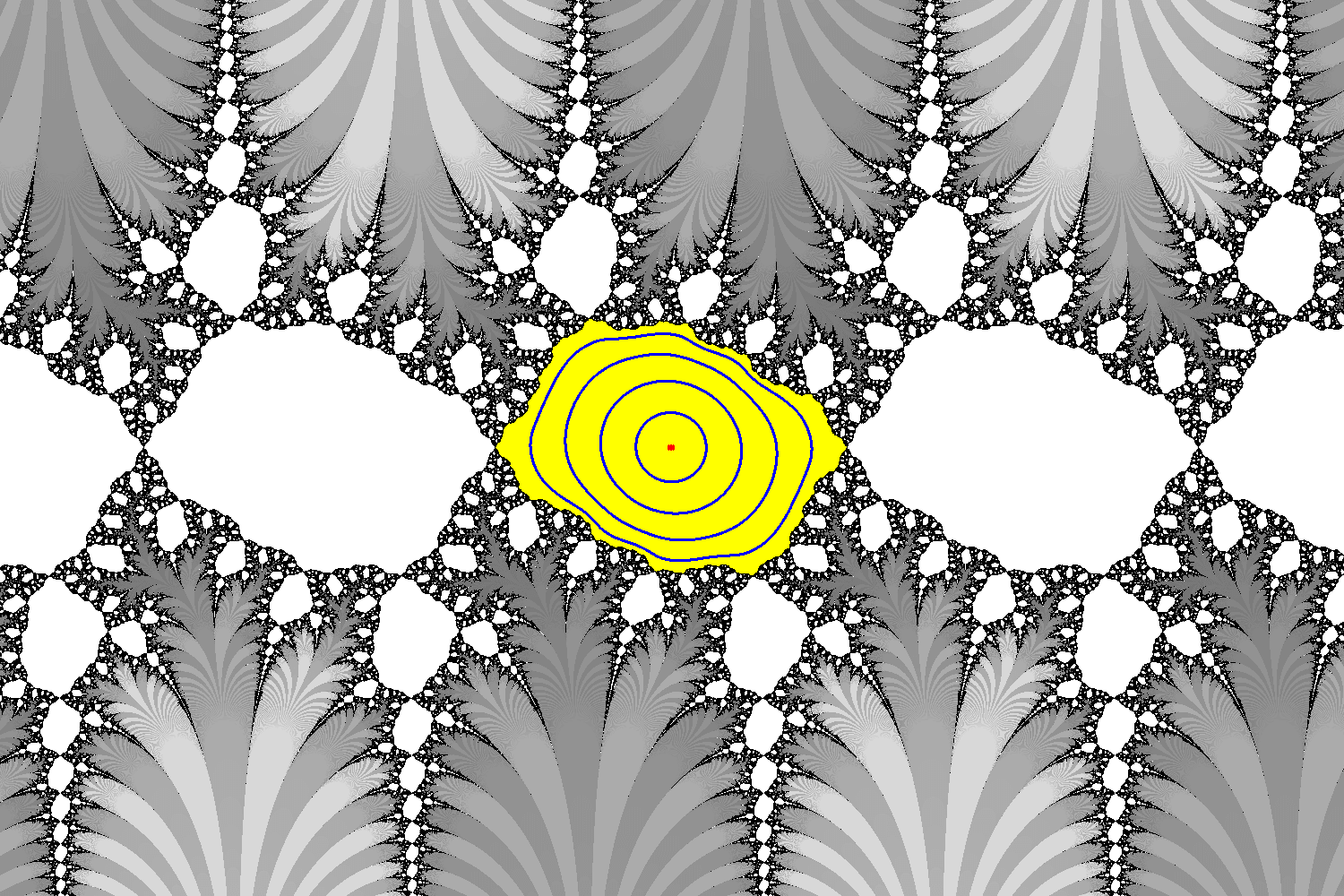}
  \caption{The dynamical plane of $S_\theta(z)=e^{2\pi\ii\theta}\sin(z)$, where $\theta=(\sqrt{5}-1)/2$ is of bounded type (courtesy of A. Ch\'{e}ritat). The Siegel disk is colored yellow with invariant curves and the rest Fatou components are colored white. The Julia set (colored black and layered gray) of $S_\theta$ is locally connected.}
  \label{Fig:Siegel-sin}
\end{figure}

\begin{cor}
For any bounded type $\theta$, the Julia set of $S_\theta(z) = e^{2 \pi \ii \theta} \sin(z)$ is locally connected.
\end{cor}

Besides the sine family, Theorem \ref{thm:entire} can be also applied to some other transcendental entire functions containing bounded type Siegel disks (see \cite{Che06}, \cite{Yan13}, \cite{CE18} and \cite{ZFS20}).

\medskip

If a transcendental entire function has an asymptotic value in the Fatou set, then the Fatou set contains an unbounded component and the Julia set is not locally connected \cite[Lemma 2.4]{BFR15}.
For $E(z)=P(z)e^{Q(z)}$, where $P$ and $Q$ are non-constant polynomials such that $E$ has a bounded type Siegel disk centered at $0$, Zakeri proved that this Siegel disk is bounded by a quasi-circle \cite{Zak10}. However, the Julia set of $E$ is not locally connected since $0$ is an asymptotic value of $E$.
By generalizing Theorem \ref{thm:entire} suitably (see Theorem \ref{thm:main-pre}), we prove that some transcendental entire functions with Siegel disks can have locally connected Julia sets even if they have asymptotic values in the Julia sets or if they are not contained in the class $\MB$.

\begin{thm}\label{thm:lc-asymp}
There exists $\lambda\in\C$ with $|\lambda|>1$ such that $f(z)=\lambda ze^z$ has a cycle of Siegel disks and a locally connected Julia set.
Moreover, $g(z)=e^z+z+\log\lambda$ has infinitely many Siegel disks and a locally connected Julia set.
\end{thm}

Note that the function $f$ in Theorem \ref{thm:lc-asymp} has an asymptotic value at $0$ (which is also a repelling fixed point) and $g$ satisfies $f(e^z)=e^{g(z)}$ but it is not contained in the class $\MB$.

For the study of the boundedness of the Siegel disks of transcendental entire functions, we also refer to \cite{Her85}, \cite{Rem04}, \cite{Rem08}, \cite{BF10} and \cite{BF18}.
In fact, to apply Theorem \ref{thm:entire} (or Theorem \ref{thm:main-pre}), it is necessary to know a priori that the Siegel disk boundaries of the considered transcendental entire function are quasi-circles.
Although some related results are established by quasiconformal surgery or polynomial-like renormalization theory, this is still a challenge in general.

\subsection{Idea of the proofs}

To obtain the local connectivity of the Julia sets of hyperbolic, (strongly) subhyperbolic and geometrically finite maps (including rational maps and transcendental entire functions), constructing uniformly expanding metrics near the Julia sets is the main strategy. For the maps which do not have uniform expanding metrics near the Julia sets, constructing puzzles is another powerful method (which is valid for some Siegel polynomials \cite{Pet96}, \cite{Pet98}, \cite{PZ04}).
For general rational maps and transcendental entire functions with Siegel disks, neither uniformly expanding metrics nor puzzles are available.

In \cite{WYZZ25}, the authors study the local connectivity of some rational maps with bounded type Siegel disks when the critical orbits are well controlled.
The proof there is based on introducing a quasi-Blaschke model and then capturing a weak expansion property near boundaries of bounded type Siegel disks under a long iteration.
In fact, the key result (see \cite[Lemma 4.5]{WYZZ25} or Lemma \ref{lem:back-1}) about the expanding property is also valid for some transcendental entire functions.

Similar to the rational case, we apply Whyburn's local connectivity criterion (see Lemma \ref{lem:LC-criterion}) to the maps $f$ in Theorem \ref{thm:entire}. Firstly, we need to prove that all Fatou components of $f$ are bounded and have locally connected boundaries. To obtain this, we introduce an orbifold metric such that $f$ is uniformly expanding in a domain $W$ whose boundary has positive distance from the boundary of every Siegel disk. The construction of the orbifold metric is strongly inspired by \cite{Mih12} and \cite{Par22}, where the strongly subhyperbolic transcendental entire functions are studied. To find the domain $W$, we use the result about the landing of some escaping subsets at periodic points in \cite{BR20} (see Theorem \ref{thm:Benini-Rempe}) to separate the boundaries of attracting basins (if any) from the boundaries of Siegel disks.
Secondly, we need to prove that the spherical diameters of Fatou components of $f$ tend to zero uniformly. This can be divided into two cases. For Siegel disks and their preimages, we use the classical shrinking lemma in the transcendental version (see Lemma \ref{lem:semi-hyperbolic}) and the weak expansion property near boundaries of bounded type Siegel disks (see Lemma \ref{lem:back-1}). For immediate attracting basins and their preimages, we rely on the uniform expanding of the orbifold metric mentioned above (see Lemma \ref{lem:expand-orbifold}).

In fact, before obtaining Theorem \ref{thm:entire}, we shall prove an intermediate result (Theorem \ref{thm:main-pre}) about the weak expansion property near boundaries of bounded type Siegel disks, which is essentially included in \cite{WYZZ25}. In particular, this result can be applied to more general transcendental entire functions with Siegel disks. For example, in the last section, we consider a family of transcendental entire functions containing a fixed asymptotic value and prove that the bounded type Siegel disks are quasi-circles and the Julia sets are locally connected (Theorem \ref{thm:lc-asymp}).

\medskip
\noindent\textbf{Notations.}
We use $\mathbb{C}$ and $\widehat{\mathbb{C}}=\mathbb{C}\cup\{\infty\}$ to denote the complex plane and the Riemann sphere respectively. Let $\D(a,r)=\{z\in\C:|z-a|<r\}$ be the Euclidean disk centered at $a\in\C$ with $r>0$. In particular, $\D:=\D(0,1)$ is the unit disk and $\T:=\partial\D$ is the unit circle.

Let $X$ be $\C$, $\EC$ or a hyperbolic domain in $\C$. We use $\dist_X(\cdot,\cdot)$ and $\diam_X(\cdot)$, respectively, to denote the distance and diameter with respect to the Euclidean, spherical or hyperbolic metric.
Let $A,B$ be two subsets in $\C$. We say that $A$ is compactly contained in $B$ if the closure of $A$ is compact and contained in the interior $\Int(B)$ of $B$ and we denote it by $A\Subset B$.

\medskip
\noindent\textbf{Acknowledgements.}
This work was supported by NSFC (Grant Nos.\,12222107, 12171276) and NSF of Shandong Province (Grant No.\,ZR2023MA044).

\section{Weak expanding and quasi-Blaschke models}\label{sec:general-quasi-Blaschke}

In this section we first state a result (Theorem \ref{thm:main-pre}) about the weak expansion property near the Siegel disk boundaries for some entire functions. Next we prepare some basic facts about pulling backs of Jordan disks and introduce the quasi-Blaschke models for the functions in Theorem \ref{thm:main-pre}. Finally we recall the half hyperbolic neighborhoods in \cite{WYZZ25} and use some results therein to prove Theorem \ref{thm:main-pre}.

\subsection{The weak expansion property}

The holomorphic functions cannot be uniformly expanding near the Siegel disk boundaries. However, it turns out that some kind of weak expansion property can be captured in certain cases.
The following result is the key ingredient in the proof of Theorem \ref{thm:entire}.

\begin{thm}\label{thm:main-pre}
Let $f$ be a transcendental entire function having a fixed bounded type Siegel disk $\Delta$ with quasi-circle boundary in $\C$. Suppose that
\begin{enumerate}
  \item $\dist_{\C}(\MP(f)\setminus\partial{\Delta},\partial{\Delta})>0$;
  \item there are only finitely many critical values and no asymptotic values whose forward orbits intersect $\partial\Delta$; and
  \item the local degrees of the critical points whose forward orbits intersect $\partial\Delta$ have a uniform upper bound.
\end{enumerate}
Then for any $\varepsilon>0$ and bounded Jordan disk $V_0\subset \C\setminus\overline{\Delta}$ with $\overline{V}_0 \cap \MP(f)\neq\emptyset$ and $\overline{V}_0 \cap \MP(f)\subset \partial\Delta$, there exists $N \geqslant 1$, such that $\diam_{\EC}(V_n)<\varepsilon$ for all $n\geqslant N$, where $V_n$ is any connected component of $f^{-n}(V_0)$.
\end{thm}

The proof of Theorem \ref{thm:main-pre} will be given in \S\ref{subsec:pf-thm-gene}.
Note that in Theorem \ref{thm:main-pre}, the function $f$ is not necessarily contained in the class $\MB$, the Julia set of $f$ is not necessarily connected and $f$ is allowed to have asymptotic values. We require $f\in\MB$ in Theorem \ref{thm:entire} since we need to consider the attracting basins and prove that $f$ is uniformly expanding with respect to certain orbifold metric in some domain (see \S\ref{sec:thm-pf}). The shrinking property under the inverse iterations of $f$ is only considered in the backward orbits of some outer half-neighborhoods of the Siegel disk boundary. Moreover, the condition $\emptyset\neq\overline{V}_0 \cap \MP(f)\subset \partial\Delta$ is non-empty since according to \cite{GS03}, $\partial\Delta$ contains at least one (recurrent) critical point and $\partial\Delta\subset \MP(f)$.



\subsection{Pulling backs and the shrinking lemma}

For a transcendental entire function $f$, recall that $S(f)$ is the set of singular values of $f$.
The following result is useful when we consider the pullbacks of simply connected domains (especially, Jordan disks) under $f$.
For a proof, see \cite[Propositions 2.8 and 2.9]{BFR15}.

\begin{prop}\label{prop:take-pre}
Let $f$ be a transcendental entire function, $V \subset \mathbb{C}$ be a simply connected domain, and $\widetilde{V}$ be a component of $f^{-1}(V)$. Then either
\begin{itemize}
\item[(1)] $f: \widetilde{V} \rightarrow V$ is a proper map and hence has finite degree; or
\item[(2)] $f^{-1}(w) \cap \widetilde{V}$ is infinite for every $w \in V$, with at most one exception. In this case, either $\widetilde{V}$ contains an asymptotic curve corresponding to an asymptotic value in $V$, or $\widetilde{V}$ contains infinitely many critical points.
\end{itemize}
If in addition $V \cap S(f)$ is compact, we have
\begin{itemize}
\item[(A)] if $\sharp(V \cap S(f)) \leqslant 1$, then $\widetilde{V}$ contains at most one critical point of $f$;
\item[(B)] in case (1), if $V$ is a bounded Jordan disk such that $\partial V \cap S(f)=\emptyset$, then $\widetilde{V}$ is also a bounded Jordan disk;
\item[(C)] in case (2), the point $\infty$ is accessible from $\widetilde{V}$.
\end{itemize}
\end{prop}

We need to consider the pullbacks of Jordan disks attaching at the Siegel disk boundaries. Hence it is necessary to consider the preimages of the bounded Jordan domain whose boundary contains singular values.

\begin{lem}\label{lem:take-pre}
Let $f$ be a transcendental entire function, $V$ be a bounded Jordan disk, and $\widetilde{V}$ be a component of $f^{-1}(V)$. Suppose
$\dist_{\C}(S(f)\setminus\partial V,\partial V)>0$, $\overline{V}$ contains no asymptotic values and $\partial V$ contains at most finitely many critical values.
If $f: \widetilde{V} \rightarrow V$ is proper, then $\widetilde{V}$ is a bounded Jordan disk. In particular, if $\widetilde{V}$ contains at most finitely many critical points or $\sharp(V \cap S(f)) \leqslant 1$, then $f: \widetilde{V} \rightarrow V$ is proper and it is conformal if $V \cap S(f)=\emptyset$.
\end{lem}


\begin{proof}
Since $\dist_{\C}(S(f)\setminus\partial V,\partial V)>0$ and $\partial V$ contains no asymptotic values, for each $z\in \partial V$ and $\widetilde{z}\in\partial \widetilde{V}$ satisfying $f(\widetilde{z})=z$, there exist open neighborhoods $W_z$ of $z$ and $\widetilde{W}_{\widetilde{z}}$ of $\widetilde{z}$ such that $f:\widetilde{W}_{\widetilde{z}}\to W_z$ is a branched covering map. This implies that $\partial \widetilde{V}$ is locally connected since $\partial V$ is a Jordan curve in $\C$.

Suppose $f: \widetilde{V} \rightarrow V$ is proper. Let $1\leqslant d<\infty$ be the degree of $f|_{\widetilde{V}}$.
If $\partial V$ contains no critical values, then $\widetilde{V}$ is a bounded Jordan disk by Proposition \ref{prop:take-pre}(B).
If $\partial V$ contains exactly $1\leqslant m<\infty$ critical values $v_1,\cdots, v_{m}$ (without counting multiplicity), then for each component $L_0$ of $\partial V\setminus\{v_1,\cdots, v_{m}\}$, $f^{-1}(L_0)\cap\partial \widetilde{V}$ consists of $d$ bounded Jordan arcs. Thus $\partial \widetilde{V}$ and hence $\widetilde{V}$ are bounded.
Since $\partial \widetilde{V}$ is locally connected and $V$ is a bounded Jordan disk, by the maximum principle, we conclude that $\widetilde{V}$ is also a bounded Jordan disk.

If $\sharp(V \cap S(f)) \leqslant 1$, by Proposition \ref{prop:take-pre}(A), $\widetilde{V}$ contains at most one critical point of $f$. Since $\overline{V}$ contains no asymptotic values, by Proposition \ref{prop:take-pre}(1) and (2), if $\widetilde{V}$ contains at most finitely many critical points or $\sharp(V \cap S(f)) \leqslant 1$, then $f: \widetilde{V} \rightarrow V$ is proper.
If $V \cap S(f)=\emptyset$, then the degree of $f: \widetilde{V} \rightarrow V$ is one and hence it is conformal and $f: \overline{\widetilde{V} }\rightarrow \overline{V}$ is a homeomorphism.
\end{proof}

Let $f: \C\to \C$ be a transcendental entire function.
For a domain $V_0$ in $\C$, every connected component of $f^{-1}(V_0)$ is called a \textit{pullback} of $V_0$. A sequence $\{V_n\}_{n\geqslant 0}$ is called \textit{a pullback sequence} of $V_0$ under $f$ if $V_{n+1}$ is a connected component of $f^{-1}(V_n)$ for all $n\geqslant 0$.
A pullback sequence $\{V_n\}_{n\geqslant 0}$ of $f$ is called \textit{semi-hyperbolic} if there exists $D_0\geqslant 1$ such that $f^{\circ n}: V_n\to V_0$ is a proper map of degree at most $D_0$ for all $n\geqslant 1$.
For transcendental entire functions, the terminology ``semi-hyperbolic" was firstly used in \cite{BM02} (see also \cite{RS11a} for the further study of such kind functions).
The following result is essentially contained in \cite[Theorem 1]{BM02}. But for convenience, we include a proof here by a slightly different argument.

\begin{lem}\label{lem:semi-hyperbolic}
Let $f$ be a transcendental entire function, and $U_0, V_0$ be two bounded Jordan disks and $D_0\geqslant 1$. Suppose $U_0$ is not contained in any Siegel disk of $f$ and $V_0\Subset U_0$.
Then for any $\varepsilon>0$, there exists $N>0$ such that for any semi-hyperbolic pullback sequence $\{U_n\}_{n\geqslant 0}$ of degree at most $D_0$, $\diam_{\EC}(V_n)<\varepsilon$ for all $n\geqslant N$, where $V_n$ is any connected component of $f^{-n}(V_0)$ contained in $U_n$.
\end{lem}

\begin{proof}
Suppose this lemma is false. Then there exist $\varepsilon_0>0$, a sequence $(n_k)_{k\geqslant 0}$ tending to $\infty$ and connected components $V_{n_k}$ in $U_{n_k}$ such that $\diam_{\EC}(V_{n_k})\geqslant \varepsilon_0$ for all $k$. Hence there exists $R>0$ such that $\D(0,R)\cap V_{n_k}\neq\emptyset$ for all $k$. Take $v_k\in \D(0,R)\cap V_{n_k}$. Since $U_{n_k}$ is simply connected, there exists a conformal map $\phi_k:\D\to U_{n_k}$ with $\phi_k(0)=v_k$.

By Koebe's distortion theorem, we conclude that $\{\phi_k\}_{k\geqslant 0}$ forms a normal family since $U_{n_k}$ does not contain a point $z_0$ for all $k$, where $f^{\circ n}(z_0)\to\infty$ as $n\to\infty$ and $f^{\circ n}(z_0)\not\in U_0$ for all $n\geqslant 0$ (the existence of such escaping point $z_0$ is guaranteed by \cite{Ere89}). Passing to subsequences if necessary, we assume that $\phi_k$ converges to $\phi:\D\to\C$ and $B_k:=f^{\circ n_k}\circ\phi_k:\D\to U_0$ converges to $B:\D\to \overline{U}_0$ locally uniformly as $k\to\infty$. Based on the fact that $B_k:\D\to U_0$ is a proper holomorphic map of degree at most $D_0$ with $B_k(0)\in V_0$, we deduce that $|\phi_k'(0)|$ has a uniformly lower bound away from $0$, which implies that $\phi$ is not a constant and hence univalent. Then $f^{\circ n_k}=B_k\circ \phi_k^{-1}$ converges locally uniformly to a non-constant holomorphic map $B\circ\phi^{-1}$ in $\phi(\D)$.

There exists a bounded Jordan disk $W\Subset\phi(\D)$ and a large integer $k_0\geqslant 1$ such that
\begin{equation}\label{equ:inclusion}
V_0\subset f^{\circ n_k}(W)\subset U_0 \text{\quad for all }k\geqslant k_0.
\end{equation}
If $U_0\cap J(f)\neq\emptyset$, then without loss of generality we assume that $V_0$ is large such that $V_0\cap J(f)\neq\emptyset$.
By \cite{Bak68}, $J(f)$ is the closure of the set of repelling periodic points of $f$. This implies that the sequence $\{f^{\circ n}\}_{n\geqslant 0}$ is normal at some point $z\in\C$ if and only if its some subsequence is (see \cite[Theorem 3.13, p.\,55]{HY98}).
Therefore, $\{f^{\circ n_k}\}_{k\geqslant k_0}$ is not normal in $W$ and $\bigcup_{k\geqslant k_0}f^{\circ n_k}(W)$ omits at most two points, which contradicts \eqref{equ:inclusion}.
If $U_0\cap J(f)=\emptyset$. Then $U_0$ is contained in a Fatou component $\widetilde{U}_0$ of $f$. By \eqref{equ:inclusion}, $\widetilde{U}_0$ is neither an attracting basin, a parabolic basin, a Baker domain nor a wandering domain. Note that we have assumed that $U_0$ is not contained in any Siegel disk. This contradiction implies the lemma.
\end{proof}

If $f$ is rational, Lemma \ref{lem:semi-hyperbolic} is refereed as the \textit{classical Shrinking Lemma} (replacing ``Siegel disk" by ``rotation domain", see \cite{Man93}, \cite{TY96} and \cite[p.\,86]{LM97}).

\subsection{Quasi-Blaschke models}\label{subsec:quasi-Blaschke}

In the rest of this section, we fix the transcendental entire function $f$ in Theorem \ref{thm:main-pre} which has a fixed bounded type Siegel disk $\Delta$ with quasi-circle boundary in $\C$.
According to \cite{GS03}, $\partial\Delta$ contains at least one critical point of $f$.
Instead of looking for an analytic Blaschke model of $f$, we construct a \textit{quasi-Blaschke model} $G$ of $f$ which is quasi-regular, such that the dynamics of $G$ is symmetric about the unit circle $\T$ and that $f$ is conjugate to $G$ outside the first preimage of $\Delta$. Then the weak expansion property of $f$ in Theorem \ref{thm:main-pre} is reduced to that of $G$.
We would like to mention that most settings in this subsection are strongly inspired by \cite{WYZZ25}.

Let
\begin{equation}\label{equ:qc-1}
\phi:\EC\setminus\overline{\Delta}\to\EC\setminus\overline{\D}
\end{equation}
be a conformal isomorphism fixing $\infty$.
For $z\in\EC$ (resp. $Z\subset\EC$), let $z^*=1/\overline{z}$ (resp. $Z^*=\{z^*: z\in Z\}$) be the symmetric image of $z$ (resp. $Z$) about the unit circle $\T$.
Let $\MP(f)$ the postsingular set of $f$.
We extend $\phi$ to a quasiconformal homeomorphism $\phi:\EC\to\EC$ as the following.
By the assumption $\dist_{\C}(\MP(f)\setminus\partial{\Delta},\partial{\Delta})>0$ in Theorem \ref{thm:main-pre}, there exist two bounded Jordan disks $U_0$ and $U_1$ such that \begin{equation}
\MP(f)\cap\Delta\subset U_0, \quad\overline{U}_0\subset\Delta,\quad \overline{U}_1\subset\C\setminus\overline{\Delta}\text{\quad and\quad} f(U_1)\subset U_0.
\end{equation}
Note that $\phi(U_1)$ is a Jordan disk in $\C\setminus\overline{\D}$.
Let $\phi:\EC\to\EC$ be a quasiconformal extension of \eqref{equ:qc-1} such that
\begin{equation}\label{equ:phi-int}
\phi(\overline{U}_0)\subset (\phi(U_1))^*\subset\D.
\end{equation}

We define a \textit{quasi-Blaschke model} corresponding to $f$:
\begin{equation}\label{equ:G}
G(z):=\left\{
\begin{array}{ll}
\phi \circ f \circ \phi^{-1}(z) & \text{\quad if } z\in\C\setminus\D,\\
\big(\phi \circ f \circ \phi^{-1}(z^*)\big)^* &\text{\quad if } z\in\D\setminus\{0\}.
\end{array}
\right.
\end{equation}
By the construction and the assumption $\dist_{\C}(\MP(f)\setminus\partial{\Delta},\partial{\Delta})>0$, the map $G$ has the following properties:
\begin{itemize}
\item[(i)] $G$ commutes with $z\mapsto z^*$, i.e., $G(z^*)=\big(G(z)\big)^*$;
\item[(ii)] $G:\C\setminus\D\to\C$ is conjugate to $f:\C\setminus\Delta\to\C$ by the quasiconformal mapping $\phi^{-1}:\EC\to\EC$;
\item[(iii)] $G:\C\setminus\{0\}\to\EC$ is a quasi-regular map and analytic in $\C\setminus (Q\cup Q^* \cup\{0\})$, where $Q=\phi\big(\overline{f^{-1}(\Delta)\setminus\Delta}\big)$; and
\item[(iv)] $\dist_{\C}(\widetilde{\MP}(G)\setminus\T,\T)>0$, where
\begin{equation}\label{equ:P-G-tilde}
\widetilde{\MP}(G):=\overline{\bigcup_{n\geqslant 0} G^{\circ n}\big(\phi(S(f))\cup \phi(S(f))^*\big)}.
\end{equation}
\end{itemize}
Indeed, (iv) holds since $S(f)\cap\Delta\subset U_0$ and by \eqref{equ:phi-int}, we have
\begin{equation}\label{equ:phi-inclusion}
\phi(U_0)\subset\phi(U_1)^* \overset{G}{\longrightarrow} (\phi\circ f(U_1))^*\subset \phi(U_0)^*.
\end{equation}
By the property (i), we have $G^{\circ 2}(\phi(U_0))\subset \phi(U_0)$.
The property (iv) guarantees that the domain in \eqref{equ:Omega-I} and the set in \eqref{equ:H-d-I} are well defined.
Denote
\begin{equation}\label{equ:P-G}
\MP(G):=\phi\big(\MP(f)\setminus\Delta\big).
\end{equation}
Then $\T\subset \MP(G)\subset\widetilde{\MP}(G)$ since $\partial\Delta\subset \MP(f)$.
Let $CV(f)$ be the set of all critical values of $f$. We denote
\begin{equation}\label{equ:crit-T}
\widetilde{\MCV}:=~\left\{G^{\circ n}(v)
\left|
\begin{array}{l}
v\in\phi(CV(f))\text{ and there exists a} \\
\text{minimal } n\geqslant 0 \text{ such that } G^{\circ n}(v)\in\T
\end{array}
\right.
\right\}.
\end{equation}
By \eqref{equ:phi-inclusion} and the assumptions in Theorem \ref{thm:main-pre}, the set $\widetilde{\MCV}\subset\T$ is finite.


\medskip
Let $V_0$ be a bounded Jordan disk in $\C\setminus\overline{\D}$. We call $\{V_n\}_{n\geqslant 0}$ a \textit{pullback sequence} of $V_0$ if $V_{n+1}$ is a connected component of $G^{-1}(V_n)$ in $\C\setminus\overline{\D}$  for all $n\geqslant 0$.
The following result is an immediate consequence of Lemma \ref{lem:take-pre}.

\begin{cor}\label{cor:take-pre}
Let $V_0$ be a bounded Jordan disk in $\C\setminus\overline{\D}$ with $\overline{V}_0 \cap \MP(G) \subset \T$. Then for any pullback sequence $\{V_n\}_{n\geqslant 0}$ in $\C\setminus\overline{\D}$, every $V_n$ is a bounded Jordan disk and $G^{\circ n}:V_n\to V_0$ is conformal.
\end{cor}

The following result implies Theorem \ref{thm:main-pre} and the proof will be given in \S\ref{subsec:pf-thm-gene}.

\begin{prop}\label{prop:thm-pre}
For any $\varepsilon>0$ and any bounded Jordan disk $V_0\subset \C\setminus\overline{\D}$ satisfying $\overline{V}_0 \cap \MP(G) \neq \emptyset$ and $\overline{V}_0 \cap \MP(G) \subset \T$, there exists $N \geqslant 1$, such that for any pullback sequence $\{V_n\}_{n\geqslant 0}$ of $V_0$, we have $\diam_{\EC}(V_n)<\varepsilon$ for all $n\geqslant N$.
\end{prop}

For any subarc $I\subsetneq\T$, we use $|I|\in(0,\pi)$ to denote the Euclidean length of $|I|$.
There exists $r'\in(0,\pi/4)$ such that if $|I|<r'$, then
\begin{equation}\label{equ:r-pri}
\sharp\big(\overline{I}\cap G^{\circ n}(\widetilde{\MCV})\big)\leqslant  1 \text{ for all } n\geqslant 0.
\end{equation}
Denote the finite set
\begin{equation}\label{equ:P-G-hat}
\widehat{\MP}(G):=\widetilde{\MCV}\cup\big\{z\in\MP(G)\setminus\T:\exists\, n\geqslant 1 \text{ such that } G^{\circ n}(z)\in\T\big\}.
\end{equation}

\begin{lem}\label{lem:take-pre-2}
Let $V_0\subset\C\setminus\overline{\D}$ be a bounded Jordan disk such that $\overline{V}_0\cap\MP(G)\neq\emptyset$ is a subarc $I\subset\T$ with $|I|<r'$.
Then there exists a bounded Jordan disk $U_0$ containing $\overline{V}_0$ such that
\begin{enumerate}
\item $\sharp(\partial U_0\cap\T)=2$, $U_0\subset\C\setminus(\MP(G)\setminus\T) $ and $(\overline{U}_0\setminus\overline{V}_0)\cap \widehat{\MP}(G)=\emptyset$;
\item for any component $V_1$ of $G^{-1}(V_0)$ in $\C\setminus\overline{\D}$ with $\overline{V}_1\cap\T=\emptyset$, the component $U_1$ of $G^{-1}(U_0)$ containing $V_1$ satisfies $\overline{U}_1\cap\overline{\D}=\emptyset$, $\sharp(\overline{U}_1\cap\MP(G))\leqslant 1$, $(\overline{U}_1\setminus\overline{V}_1)\cap \MP(G)=\emptyset$ and $U_1$ is a bounded Jordan disk.
\end{enumerate}
\end{lem}

\begin{proof}
Since $\widehat{\MP}(G)$ is a finite set and $\dist_{\C}(\widetilde{\MP}(G)\setminus\T,\T)>0$, the existence of a bounded Jordan disk $U_0$ satisfying (a) is immediate. For any such $U_0$, we only need to verify the property (b).
By \eqref{equ:r-pri} and (a), we have
\begin{equation}\label{equ:PG-inclusion}
\overline{U}_0\cap\widehat{\MP}(G)\subset\overline{V}_0\cap\widehat{\MP}(G)\subset\overline{I}\cap\widehat{\MP}(G)
\end{equation}
and $\sharp(\overline{I}\cap\widehat{\MP}(G))\leqslant 1$.
In particular, $U_0$ contains at most one critical value and no asymptotic values. By Proposition \ref{prop:take-pre}, $G:U_1\to U_0$ is a proper map and have finite degree $d\geqslant 1$ and $U_1$ is a bounded Jordan disk.

Without loss of generality, we assume $\sharp(\overline{I}\cap\widehat{\MP}(G))=1$ since the case $\overline{I}\cap\widehat{\MP}(G)=\emptyset$ is completely similar and easier.
Denote $\{v\}:=\overline{I}\cap\widehat{\MP}(G)$.
Since $\sharp(\partial U_0\cap\T)=2$, it implies that $J:=U_0\cap\T$ is an open arc containing $I$. Then $G^{-1}(J\setminus\{v\})\cap U_1$ consists of $2d$ open Jordan arcs $J_1$, $\cdots$, $J_{2d}$. We claim that $\overline{J}_i\cap\T=\emptyset$ for all $1\leqslant i\leqslant 2d$. Indeed, if $\overline{J}_i\cap\T\neq\emptyset$ for some $i$, then $\overline{J}_i\cap\T$ contains a critical point and we would have $\sharp(\overline{U}_0\cap\widehat{\MP}(G))\geqslant 2$, which is a contradiction. This implies that $\overline{U}_1\cap\overline{\D}=\emptyset$ and thus we have $\sharp(\overline{U}_1\cap\MP(G))\leqslant 1$ and $(\overline{U}_1\setminus\overline{V}_1)\cap \MP(G)=\emptyset$.
\end{proof}

Let $\widetilde{\MP}(G)$ be defined in \eqref{equ:P-G-tilde}. There exists $r''>1$ such that $(\widetilde{\MP}(G)\setminus\T)\cap\{z\in\C:1/r''\leqslant |z|\leqslant  r''\}=\emptyset$.
Let $I\subsetneq\T$ be an open subarc. Then $\widetilde{\MP}(G) \setminus I$ is a closed subset of $\C$. We denote
\begin{equation}\label{equ:Omega-I}
\Omega_I:= \text{The component of } \C\setminus (\widetilde{\MP}(G) \setminus I) \text{ containing } I.
\end{equation}
Then $\Omega_I$ is a domain which is symmetric about $\T$.
For any given $d>0$, let
\begin{equation}\label{equ:H-d-I}
H_d(I) := \big\{z \in \Omega_I \,:\, \dist_{\Omega_I}(z , I) < d \text{~and~} |z| > 1 \big\}
\end{equation}
be a \textit{half hyperbolic neighborhood} of $I$ in $\C\setminus\overline{\D}$.
See Figure \ref{Fig_half-nbd}.

\begin{figure}[!htpb]
  \setlength{\unitlength}{1mm}
  \centering
  \includegraphics[width=0.7\textwidth]{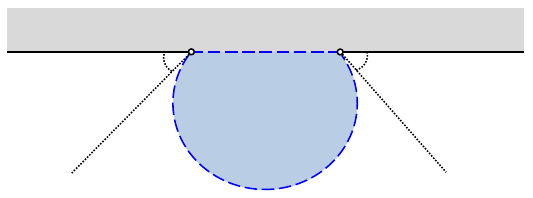}
  \put(-88,30){$\D$}
  \put(-51,29){$I$}
  \put(-73,23){$\beta$}
  \put(-30,23){$\beta$}
  \put(-55,14){$H_d(I)$}
  \caption{A half hyperbolic neighborhood $H_d(I)$ of the open arc $I$. }
  \label{Fig_half-nbd}
\end{figure}
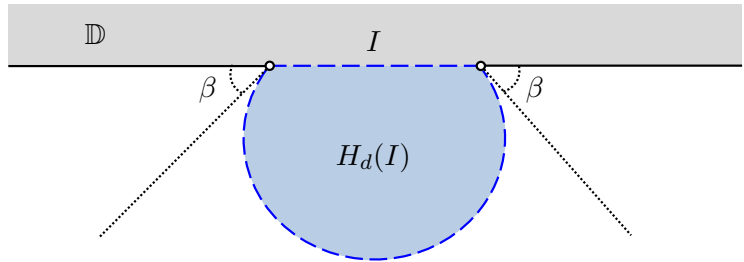

Let $I\subsetneq\T$ be an open subarc and $d>0$. Consider $\widetilde{\Omega}_I:=\EC\setminus(\T\setminus I)$ and
\begin{equation}
\widetilde{H}_d(I): = \big\{z \in \widetilde{\Omega}_I \,:\, \dist_{\widetilde{\Omega}_I}(z , I) < d\big\}.
\end{equation}
It is well known that the boundary $\partial\widetilde{H}_d(I)$ consists of two subarcs of Euclidean circles which are symmetric to $\T$ and form an angle $\beta\in(0,\pi)$
between these circles and $\T\setminus I$, where $\log\cot\big(\tfrac{\beta}{4}\big)=d$. See \cite[Chap.\,VI, \S 5]{MS93}.
Actually, the angle between $\partial H_d(I)\setminus I$ and $\T\setminus I$ is well-defined and also equal to $\beta$. See \cite[\S 2.3]{WYZZ25}.

It is proved in \cite[Lemmas 2.6 and 2.7(1)]{WYZZ25} that for any $\tilde{d}>0$, there exists a small number $r_1=r_1(\tilde{d})>0$ such that if $|I|<r_1$, then the half hyperbolic neighborhood $H_{d}(I)$ is a bounded Jordan disk in $\C\setminus\overline{\D}$ for any $0<d\leqslant \tilde{d}+1$ and
\begin{equation}\label{equ:shape-H-d-I}
H_d(I)\subset \widetilde{H}_d(I)\setminus\overline{\D}\subset H_{d+1}(I).
\end{equation}

\begin{lem}\label{lem:back-1}
There exists a number $d_0>0$ such that for any $d\geqslant d_0$ and $\varepsilon>0$, there exists $r=r(d,\varepsilon)\in(0,r_1]$ with $r_1=r_1(d)>0$ such that if $V_0=H_d(I_0)$ with $|I_0|< r$, then for any pullback sequence $\{V_n\}_{n\geqslant 0}$ of $V_0$, we have $\diam_{\EC} (V_n) < \varepsilon$ for all $n \geqslant 0$.
\end{lem}

Note that $V_0$ in Lemma \ref{lem:back-1} is a bounded Jordan disk by the definition of $r_1$ and hence each $V_n$ is a bounded Jordan disk by Corollary \ref{cor:take-pre}.
If we replace the rational maps in \cite[Main Lemma]{WYZZ25} by the functions in Theorem \ref{thm:main-pre}, then Lemma \ref{lem:back-1} can be obtained completely the same as \cite[Lemma 4.5]{WYZZ25}.
Indeed, the number $d_0>0$ is determined after \cite[Proposition 3.7]{WYZZ25} and the only place where we need to be careful about in the proof of \cite[Lemma 4.5]{WYZZ25} is the following property: there exists $D_0\geqslant 2$ such that for every $z\in\T$, there exists a small Jordan disk $W_0$ containing $z$, such that for any sequence $\{W_n\}_{n\geqslant 1}$ of pullbacks of $W_0$ in $\EC\setminus\overline{\D}$, we have $\deg(G^{\circ n}: W_n\to W_0)\leqslant  D_0$.
This property for the functions in Theorem \ref{thm:main-pre} (and hence for $G$ in \eqref{equ:G}) is guaranteed by the conditions on the critical orbits (see Lemma \ref{lem:take-pre-2}). The rest part of the proof can be followed essentially verbatim.

\subsection{Proof of Theorem \ref{thm:main-pre}}\label{subsec:pf-thm-gene}

As we have mentioned in the previous subsection, to prove Theorem \ref{thm:main-pre}, it is sufficient to prove Proposition \ref{prop:thm-pre}.
The argument is different from the rational case in \cite[\S 4.3]{WYZZ25} since the preimage of a Jordan disk under any rational map has only at most finitely many components while transcendental entire functions may have infinitely many.

\begin{proof}[Proof of Proposition \ref{prop:thm-pre}]
We divide the proof into several steps:

\medskip
\textit{Step 1. First reduction of $V_0$}.
By the assumption $\dist_{\C}(\MP(f)\setminus\partial{\Delta},\partial{\Delta})>0$ in Theorem \ref{thm:main-pre}, since $V_0$ is a bounded Jordan disk in $\C\setminus\overline{\D}$ satisfying $\emptyset\neq\overline{V}_0 \cap \MP(G) \subset \T$, it follows that $V_0$ is contained in another bounded Jordan disk $V_0'$ in $\C\setminus\overline{\D}$ such that $\overline{V_0'} \cap \MP(G)$ is the union of finitely many disjoint subarcs on $\T$ which are not singletons. Without loss of generality, we assume that $V_0=V_0'$.

\medskip
\textit{Step 2. Second reduction of $V_0$}. We fix $d\geqslant d_0$, where $d_0>0$ is the number in Lemma \ref{lem:back-1}. Let $r'>0$ be introduced in \eqref{equ:r-pri}.
By the definition of $r_1=r_1(d)>0$ and the characterization of the shape of half hyperbolic neighborhoods in \eqref{equ:shape-H-d-I}, the Jordan disk $V_0$ is contained in another bounded Jordan disk $\widetilde{V}_0\subset\C\setminus\overline{\D}$, which is the union of finitely many small Jordan disks $\{H_d(I_j):1\leqslant j\leqslant M_1\}$ and $\{W_k:1\leqslant k\leqslant M_2\}$, where $|I_j|<r_0:=\min\{r_1,r'\}$ for each $j$ and $\overline{W}_k\cap \MP(G)=\emptyset$ for each $k$. By Lemma \ref{lem:take-pre}, for any pullback sequence $\{\widehat{V}_n\}_{n\geqslant 0}$ of $\widehat{V}_0=H_d(I_j)$ or $W_k$ or $\widetilde{V}_0$, the map $G^{\circ n}:\widehat{V}_n\to \widehat{V}_0$ is conformal. By Lemma \ref{lem:semi-hyperbolic}, it suffices to prove Proposition \ref{prop:thm-pre} for $V_0:=H_d(I)$ when $|I|<r_0$.

\medskip
\textit{Step 3. First jumping off}.
Let $V_0=H_d(I)$ with $|I|<r_0$. Without loss of generality, we assume that $|I|\geqslant r$ since otherwise the proposition follows immediately by Lemma \ref{lem:back-1}.
For any pullback sequence $\{V_n\}_{n\geqslant 0}$ of $V_0$, by Lemma \ref{lem:take-pre}, each $V_n$ is a bounded Jordan disk and the map $G: \overline{V}_{n+1}\to \overline{V}_n$ is a homeomorphism. Let $\widehat{\MP}(G)$ be defined in \eqref{equ:P-G-hat}.
If $\partial V_n\cap\T\neq\emptyset$ for some $n$, then $\partial V_n\cap\T$ is an arc or a singleton, and moreover,
\begin{equation}\label{equ:sharp-V-n}
\sharp\,(\partial V_n \cap \widehat{\MP}(G))\leqslant  1 \text{\quad and\quad} (\partial V_n\setminus\T) \cap \widehat{\MP}(G)=\emptyset.
\end{equation}
By the assumptions in Theorem \ref{thm:main-pre}, there exists a constant $D_0\geqslant 1$ such that for any $z\in \widehat{\MP}(G)$, the local degree of $G^{\circ n}$ at $z_n$ is not larger than $D_0$ for all $n\geqslant 1$, where $z_n$ is any point in $G^{-n}(z)$.

Let $\MV_1$ be the set of all components of $G^{-1}(V_0)$ in $\C\setminus\overline{\D}$. Then $\MV_1=\MV_1'\cup\MV_1''$, where
\begin{itemize}
\item $\sharp\MV_1'<\infty$ and each $V_1\in\MV_1'$ satisfies $\overline{V}_1\cap\T\neq\emptyset$; and
\item $\sharp\MV_1''=\infty$ and each $V_1\in\MV_1''$ satisfies $\overline{V}_1\cap\T=\emptyset$.
\end{itemize}
By Lemma \ref{lem:take-pre-2}, there exists a bounded Jordan disk $U_0$ containing $\overline{V}_0$ such that for any $V_1\in\MV_1''$ with $\overline{V}_1\cap\T=\emptyset$, any pullback sequence $\{U_n\}_{n\geqslant 0}$ of $U_0$ with $V_1\subset U_1$ satisfies $\overline{U}_n\cap\overline{\D}=\emptyset$ and $\deg(G^{\circ n}: U_n\to U_0)\leqslant D_0$ for all $n\geqslant 1$.

Let $\varepsilon>0$ be given. By Lemma \ref{lem:semi-hyperbolic}, there exists $N_1>0$ such that $\diam_{\EC}(V_n)<\varepsilon$ for all $n\geqslant N_1$, where $\{V_n\}_{n\geqslant 0}$ is any pullback sequence of $V_0$ satisfying $V_1\in\MV_1''$.

\medskip
\textit{Step 4. Second jumping off or not}.
Let $r=r(d,{\varepsilon}/{2})>0$ be the number introduced in Lemma \ref{lem:back-1}. We consider two cases:

Case (1.1). Suppose $|\overline{V}_1\cap\T|<r$ for all $V_1\in\MV_1'$.
Then each Jordan disk $V_1\in\MV_1'$ is contained in another bounded Jordan disk $\widetilde{V}_1\subset\C\setminus\overline{\D}$, which is the union of finitely many small Jordan disks $H_d(I_1)$ and $\{W_k:1\leqslant k\leqslant M\}$, where $|I_1|<r$ and $\overline{W}_k\cap \MP(G)=\emptyset$ for each $k$.
By Lemma \ref{lem:back-1}, for any pullback sequence $\{\widetilde{V}_n\}_{n\geqslant 1}$ of $\widetilde{V}_1:=H_d(I_1)$, we have $\diam_{\EC} (\widetilde{V}_n) < \varepsilon/2$ for all $n \geqslant 1$.
By Lemma \ref{lem:semi-hyperbolic}, the spherical diameter of the $n$-th preimage of each $W_k$ tends to zero as $n\to\infty$. Since $\MV_1'$ is a finite set, there exists $N_1'>0$ such that $\diam_{\EC}(V_n)<\varepsilon$ for all $n\geqslant N_1'$, where $\{V_n\}_{n\geqslant 0}$ is any pullback sequence of $V_0$ satisfying $V_1\in\MV_1'$.
Thus in this case the proposition is proved by setting $N:=\max\{N_1,N_1'\}$.

Case (1.2). Suppose there exists $V_1\in\MV_1'$ such that $|\overline{V}_1\cap\T|\geqslant r$. We go back to the similar arguments in Step 3 as the following.
Let $\MV_2$ be the set of all components of $G^{-1}(\MV_1')$ in $\C\setminus\overline{\D}$. Then $\MV_2=\MV_2'\cup\MV_2''$, where
\begin{itemize}
\item $\sharp\MV_2'<\infty$ and each $V_2\in\MV_2'$ satisfies $\overline{V}_2\cap\T\neq\emptyset$; and
\item $\sharp\MV_2''=\infty$ and each $V_2\in\MV_2''$ satisfies $\overline{V}_2\cap\T=\emptyset$.
\end{itemize}
For each $V_1\in\MV_1'$, by \eqref{equ:sharp-V-n} and Lemma \ref{lem:take-pre-2}, there exists a bounded Jordan disk $U_1$ containing $\overline{V}_1$ such that for any $V_2\in\MV_2''$ with $\overline{V}_2\cap\T=\emptyset$, any pullback sequence $\{U_n\}_{n\geqslant 1}$ of $U_1$ with $V_2\subset U_2$ satisfies $\overline{U}_n\cap\overline{\D}=\emptyset$ and $\deg(G^{\circ (n-1)}: U_n\to U_1)\leqslant D_0$ for all $n\geqslant 2$.
By Lemma \ref{lem:semi-hyperbolic}, there exists $N_2>0$ such that $\diam_{\EC}(V_n)<\varepsilon$ for all $n\geqslant N_2$, where $\{V_n\}_{n\geqslant 0}$ is any pullback sequence of $V_0$ satisfying $V_1\in\MV_1'$ and $V_2\in\MV_2''$.

\medskip
\textit{Step 5. Consecutive jumping offs or not}.
Similarly, next we need to consider the following two cases:
\begin{itemize}
\item Case (2.1). $|\overline{V}_2\cap\T|<r$ for all $V_2\in\MV_2'$;
\item Case (2.2). There exists $V_2\in\MV_2'$ such that $|\overline{V}_2\cap\T|\geqslant r$.
\end{itemize}
Similar to Step 4, for Case (2.1), there exists $N_2'>0$ such that $\diam_{\EC}(V_n)<\varepsilon$ for all $n\geqslant N_2'$, where $\{V_n\}_{n\geqslant 0}$ is any pullback sequence of $V_0$ satisfying $V_1\in\MV_1'$ and $V_2\in\MV_2'$.
Thus in this case the proposition is proved by setting $N:=\max\{N_1,N_1',N_2,N_2'\}$.
For Case (2.2), the subsequent argument is similar to Case (1.2).

Suppose $\MV_k$, $\MV_k'$ and $\MV_k''$ are defined for all $2\leqslant k\leqslant n$ such that $\MV_k=\MV_k'\cup\MV_k''$ and $\MV_k$ is the set of all components of $G^{-1}(\MV_{k-1}')$ in $\C\setminus\overline{\D}$, where
\begin{itemize}
\item $\sharp\MV_k'<\infty$ and each $V_k\in\MV_k'$ satisfies $\overline{V}_k\cap\T\neq\emptyset$; and
\item $\sharp\MV_k''=\infty$ and each $V_k\in\MV_k''$ satisfies $\overline{V}_k\cap\T=\emptyset$.
\end{itemize}
Similarly as above, let $\MV_{n+1}$ be the set of all components of $G^{-1}(\MV_n')$ in $\C\setminus\overline{\D}$. Then $\MV_{n+1}=\MV_{n+1}'\cup\MV_{n+1}''$, where
\begin{itemize}
\item $\sharp\MV_{n+1}'<\infty$ and each $V_{n+1}\in\MV_{n+1}'$ satisfies $\overline{V}_{n+1}\cap\T\neq\emptyset$; and
\item $\sharp\MV_{n+1}''=\infty$ and each $V_{n+1}\in\MV_{n+1}''$ satisfies $\overline{V}_{n+1}\cap\T=\emptyset$.
\end{itemize}
Hence $\MV_n$, $\MV_n'$ and $\MV_n''$ are defined for all $n\geqslant 1$.

\medskip
\textit{Step 6. The arcs are eventually small and the conclusion}.
Since $G:\T\to\T$ is conjugate to the irrational rotation and there exists at least one critical point on $\T$, there exists a smallest integer $n_0\geqslant 1$ such that $|\overline{V}_{n_0}\cap\T|<r$ for any pullback sequence $\{V_n\}_{n\geqslant 0}$ of $V_0$.
Hence we have the following two cases:
\begin{itemize}
\item $V_n\in\MV_n'$ for all $1\leqslant n \leqslant n_0$;
\item $V_n\in\MV_n''$ for some $1\leqslant n \leqslant n_0$.
\end{itemize}
In both cases, similar to the arguments in Cases (1.1), (1.2) and (2.1), one can apply Lemmas \ref{lem:semi-hyperbolic} and \ref{lem:back-1} so that one can find $N \geqslant 1$, such that for any pullback sequence $\{V_n\}_{n\geqslant 0}$ of $V_0$, we have $\diam_{\EC}(V_n)<\varepsilon$ for all $n\geqslant N$.
The proof of Proposition \ref{prop:thm-pre} and hence Theorem \ref{thm:main-pre} is complete.
\end{proof}

\section{Orbifold metrics and the proof of Theorem \ref{thm:entire}} \label{sec:thm-pf}

In this section, we study the transcendental entire functions $f$ in Theorem \ref{thm:entire}. First we determine the types of Fatou components of $f$. Then an orbifold metric is introduced so that $f$ is uniformly expanding in some domain. Next we use this expanding metric to prove that all Fatou components of $f$ are bounded and finally we prove the local connectivity of the Julia set of $f$.

\subsection{Possible Fatou components}

We first determine all possible types of Fatou components of the functions in Theorem \ref{thm:entire}.

\begin{lem}\label{lem:Fatou-finite}
Let $f$ be a transcendental entire function in Theorem \ref{thm:entire}. Then
$F(f)$ is the union of finitely many bounded type Siegel disks and immediate attracting basins and their preimages.
Each Fatou component of $f$ is simply connected and $\MP(f)\cap F(f)$ is contained in a finite union of Fatou components.
Moreover, every periodic cycle in $J(f)$ is repelling.
\end{lem}

\begin{proof}
By the assumption that $\MP(f)\cap F(f)$ is compact, we conclude that $f$ has no parabolic basins and $\MP(f)\cap F(f)$ is contained in a finite union of Fatou components. In particular, $\MP(f)$ is disjoint with any wandering domain (if any) of $f$.
Since each immediate attracting cycle contains at least one singular value, it follows that $f$ contains at most finitely many attracting basins.
Since every Siegel disk of $f$ is of bounded type having a quasi-circle boundary in $\C$, according to \cite{GS03}, each cycle of Siegel disks contains at least one critical point and hence at least one critical value on its boundary. By the finiteness of $CV(f)\cap J(f)$, it follows that $f$ has only finitely many Siegel disks and they are all of bounded type.
By \cite[Corollary 14.4]{Mil06}, every irrationally indifferent point in $J(f)$ is accumulated by different points of the postsingular set $\MP(f)$.
Since $\MP(f)\cap J(f)$ is finite outside the union of all Siegel disk boundaries, we conclude that $J(f)$ contains no irrationally indifferent point and hence every periodic cycle in $J(f)$ is repelling.

By \cite[Theorem 1]{EL92}, $f$ has no Baker domains since $f\in\MB$. To exclude the existence of wandering domains, we apply the quasiconformal surgery.
Let $\widehat{F}(f)$ be the union of all Siegel disks and immediate attracting basins of $f$ and their preimages.
By the maximum principle, each periodic Fatou component of $f$ is simply connected.
Since $f$ has no asymptotic values and the number of critical points in each Fatou component is finite, by Proposition \ref{prop:take-pre}, every  component of $\widehat{F}(f)$ is simply connected.
By \cite[Theorem 7.34, p.\,256]{BF14a}, there exist a quasiregular map $H:\C\to\C$ which coincides with $f$ in $\C\setminus K$, where $K$ is a compact subset of $\widehat{F}(f)$, and a quasiconformal mapping $\varphi:\C\to\C$ such that $h=\varphi\circ H\circ\varphi^{-1}:\C\to\C$ is a transcendental entire function and $h$ is postsingularly finite in the Fatou set.
Hence $h$ has only finitely many singular values in $\C$. According to \cite{GK86} or \cite{EL92}, $h$ has no wandering domains and hence $f$ is also.
\end{proof}

\begin{rmk}
For the non-existence of wandering domains of $f$, one can also use another argument as the following.
Assume that $U$ is a wandering domain of $f$. By \cite{BHKMT93}, the limit function of $\{f^{\circ n}|_U\}$ are contained in $\MP(f)'\cup\{\infty\}$, where $\MP(f)'$ is the derived set of $\MP(f)$.
It is well known that any limit function  of $\{f^{\circ n}|_U\}$ is a constant.
Since $f\in\MB$, we cannot have $f^{\circ n}|_U\to\infty$ \cite{EL92}. By the finiteness of $CV(f)\cap J(f)$, there exists a subsequence $\{f^{\circ n_k}|_U\}$ whose limit is a finite point $a$, which must lie on the boundary of some bounded type Siegel disk $\Delta$. Let $\widetilde{V}_0$ be an open neighborhood of $a$ such that $V_0:=\widetilde{V}_0\setminus\overline{\Delta}$ is bounded Jordan disk satisfying $\overline{V}_0 \cap \MP(f)\subset \partial\Delta$. By Theorem \ref{thm:main-pre}, for any $\varepsilon>0$, there exists $N \geqslant 1$, such that $\diam_{\EC}(V_n)<\varepsilon$ for all $n\geqslant N$, where $V_n$ is any connected component of $f^{-n}(V_0)$. This implies that $\diam_{\EC}(U)$ should be arbitrarily small, which is a contradiction.
\end{rmk}

For any transcendental entire function $f$, we have $\MP(f)=\MP(f^{\circ n})$ for any $n\geqslant 1$. If $f$ is a function  in Theorem \ref{thm:entire}, then it is straightforward to verify that $f^{\circ n}$ is also (based on Proposition \ref{prop:take-pre}).
Our aim is to study the local connectivity of $J(f)$.
Based on Lemma \ref{lem:Fatou-finite}, performing a quasiconformal surgery (see \cite[Theorem 7.34, p.\,256]{BF14a}) and iterating $f$ finitely many times if necessary, we assume that each transcendental entire function $f$ in Theorem \ref{thm:entire} satisfies
\begin{itemize}
\item all periodic Fatou components of $f$ have period one; and
\item $f$ is postsingularly finite in the Fatou set.
\end{itemize}
Suppose $f$ has exactly $p\geqslant 1$ fixed bounded type Siegel disks $\Delta_1,\cdots,\Delta_{p}$ and $q\geqslant 0$ fixed immediate attracting basins $A_1,\cdots, A_{q}$ (note that $f$ may have no attracting basins). There exist bounded Jordan disks $B_j\Subset A_j$ such that
\begin{equation}\label{equ:P-compact}
\MP(f)\cap A_j\subset B_j \text{\quad and\quad} f(B_j)\Subset B_j
\end{equation}
for all  $1\leqslant j\leqslant q$.
Denote
\begin{equation}\label{equ:W}
W:=\C\setminus\Big(\bigcup_{i=1}^{p} \overline{\Delta}_i\cup \bigcup_{j=1}^{q}\overline{B}_j\Big).
\end{equation}
Then $f^{-1}(W)\subset W$ and $\MP(f)\cap W$ is a finite set. Note that every $\partial\Delta_i$ is a quasi-circle and $\overline{\Delta}_{i_1}\cap\overline{\Delta}_{i_2}=\emptyset$ for any integers $1\leqslant i_1\neq i_2\leqslant p$.

\subsection{Expanding orbifold metrics}

Let $\N_+:=\{1,2,3,\cdots\}$ be the set of all positive integers. The following definition was introduced in \cite{Thu84} (see also \cite[Appendix A.2]{McM94b} and \cite[Appendix E]{Mil06}).

\begin{defi}[Orbifolds]
A \textit{Riemann orbifold} (\textit{orbifold} in short) is a pair $\MO=(S,\nu)$ consisting of an \textit{underlying} Riemann surface $S$ and a ramification function $\nu:S\to\N_+$ which takes the value $\nu(z)=1$ except on a discrete closed subset. A point $z\in S$ is called \textit{ramified} if $\nu(z)\geqslant 2$ and \textit{unramified} otherwise.
\end{defi}

Suppose $h:\widetilde{S}\to S$ is a holomorphic map between two Riemann surfaces, and $\widetilde{\MO}=(\widetilde{S},\widetilde{\nu})$ and $\MO=(S,\nu)$ are orbifolds. Then
\begin{itemize}
\item $h:\widetilde{\MO}\to \MO$ is \textit{holomorphic} if $\nu(h(w))$ divides $\deg_{w}(h)\,\widetilde{\nu}(w)$ for all $w\in \widetilde{S}$, where $\deg_{w}(h)$ is the local degree of $h$ at $w$;
\item $h:\widetilde{\MO}\to \MO$ is an \textit{orbifold covering} if $h:\widetilde{S}\to S$ is a branched covering and $\nu(h(w))=\deg_{w}(h)\,\widetilde{\nu}(w)$ for all $w\in \widetilde{S}$; and
\item if $h:\widetilde{\MO}\to \MO$ is an orbifold covering, $\widetilde{S}$ is simply connected and $\widetilde{\nu}\equiv 1$ on $\widetilde{S}$, then $h$ is called a \textit{universal covering}.
\end{itemize}

With only two exceptions, each orbifold $\MO=(S,\nu)$ has a universal covering surface $\widetilde{S}$ which is conformally equivalent to one of $\EC$, $\C$ and $\D$. In the last case the orbifold is called \textit{hyperbolic}, and we only consider such type of orbifolds in this paper. In particular, if $S$ is a hyperbolic Riemann surface, then $\MO$ is a hyperbolic orbifold (see \cite[Theorem A.2]{McM94b}) and the universal covering $h:\D\to \MO$ descends the hyperbolic metric $\rho_{\D}(w)|dw|$ in $\D$ to the \textit{orbifold metric} $\rho_{\MO}(z)|dz|$ in $S$ which satisfies $\rho_{\MO}(h(w))|h'(w)|=\rho_{\D}(w)$. If $S\subset\C$, then the induced orbifold metric $\rho_{\MO}(z)|dz|$ is topological equivalent to the Euclidean metric in $S$.

The well-known Schwarz-Pick theorem for hyperbolic Riemann surfaces generalizes to the setting of hyperbolic orbifolds \cite[Proposition 17.4]{Thu84} (see also \cite[Theorem A.3]{McM94b}). In particular, it has the following consequence.

\begin{lem}\label{lem:Thurston}
Let $\widetilde{\MO}=(\widetilde{S},\widetilde{\nu})$ and $\MO=(S,\nu)$ be two hyperbolic orbifolds such that $\widetilde{S}\subset S$. Suppose $h: \widetilde{\MO}\to \MO$ is an orbifold covering map, and that the inclusion $\widetilde{\MO}\hookrightarrow \MO$ is holomorphic but not an orbifold covering. Then for all unramified $z \in \widetilde{S}$,
\begin{equation}
\|D h(z)\|_{\MO}:=\frac{\rho_{\widetilde{\MO}}(z)}{\rho_{\MO}(z)}=\frac{\rho_{\MO}(h(z))\,|h'(z)|}{\rho_{\MO}(z)}>1.
\end{equation}
\end{lem}

In the following, we construct a hyperbolic orbifold induced by $f$ and the idea is strongly inspired by \cite[\S 3]{Mih12} and \cite[\S 5]{Par22}.
By \cite{Bak68}, repelling periodic points of $f$ are dense in $J(f)$. Let $\MC\subset J(f)$ be a cycle of repelling periodic points of $f$ such that $\MC\cap\MP(f)=\emptyset$. Then we have $\MC\Subset W$, where $W$ is the hyperbolic Riemann surface introduced in \eqref{equ:W}. Define
\begin{equation}\label{equ:nu}
\nu(z):=
\left\{
\begin{array}{ll}
2  &~~~~~~~\text{if}~ z\in\MC, \\
\text{lcm} \big\{\deg_{w}(f^{\circ n}): f^{\circ n}(w)=z \text{ for some } n\geqslant 1\big\} &~~~~~~\text{if}~ z\in W\setminus\MC.
\end{array}
\right.
\end{equation}
Since $\MP(f)\cap W$ is finite and the local degrees of all critical points of $f$ have a uniform upper bound, we conclude that $\nu: W\to\N_+$ is well-defined. Then
\begin{equation}\label{equ:O-f}
\MO_f:=(W,\nu)
\end{equation}
is a hyperbolic orbifold.
Let $\widetilde{W}$ be any connected component of $f^{-1}(W)$. Then $\widetilde{W}$ is a proper subset of $W$. The hyperbolic orbifold
\begin{equation}\label{equ:nu-tilde}
\widetilde{\MO}_f:=(\widetilde{W},\widetilde{\nu}) \text{\quad is defined by\quad}\widetilde{\nu}(z):=\frac{\nu(f(z))}{\deg_z(f)} \text{ for }z\in \widetilde{W}.
\end{equation}
Since $f : \widetilde{W}\to W$ is a branched covering, based on \eqref{equ:nu} and \eqref{equ:nu-tilde}, it is easy to verify that $f:\widetilde{\MO}_f\to \MO_f$ is an orbifold covering map. Moreover, the inclusion $\widetilde{\MO}_f\hookrightarrow \MO_f$ is holomorphic but not an orbifold covering.

\medskip
By \cite[Proof of Lemma 5.1]{Rem09b} (see also \cite[Proof of Proposition 3.4]{Mih10}), the assumption that $f\in\MB$ in Theorem \ref{thm:entire} implies that there exist a constant $K>1$ and a sequence $\{z_i\}_{i\geqslant 0}\subset \widetilde{W}$ such that
\begin{equation}
f(z_i)\in\MC, \quad \widetilde{\nu}(z_i)=2 \text{\quad and\quad} |z_i|<|z_{i+1}|\leqslant K|z_i| \text{\quad for all } i\geqslant 0.
\end{equation}
By \cite[\S\S4.2 and 4.3]{Mih12} (see also \cite[\S 5]{Par22} for more precise estimate), we have
\begin{equation}
\frac{1}{\rho_{\widetilde{\MO}_f}(z)}\leqslant O(|z|) \text{\quad and\quad}\rho_{\MO_f}(z)\leqslant O\Big(\frac{1}{|z|\log|z|}\Big)
\end{equation}
as $z\to\infty$ in $\widetilde{W}$ and $W$ respectively. This implies that
\begin{equation}\label{equ:expand-infty}
\frac{\rho_{\widetilde{\MO}_f}(z)}{\rho_{\MO_f}(z)}\to\infty \text{\quad as } z\to\infty \text{ in }\widetilde{W}.
\end{equation}

By \eqref{equ:P-compact}, we have $\partial \widetilde{W}\setminus\big(\bigcup^{p}_{i=1}\partial\Delta_i\big)\subset W$.
For any $\delta>0$, we denote
\begin{equation}\label{equ:W-tilde-delta}
\widetilde{W}_\delta:=\{z\in\widetilde{W}: \dist_{\C}\Big(z,\,\,\bigcup^{p}_{i=1}\overline{\Delta}_i\Big)>\delta\}.
\end{equation}
Based on \eqref{equ:expand-infty} and Lemma \ref{lem:Thurston}, we have the following immediate result.

\begin{lem}[Uniform expansion]\label{lem:expand-orbifold}
For every $\delta>0$, there exists $\lambda=\lambda(\delta)>1$ such that for all unramified $z \in \widetilde{W}_\delta$, we have
\begin{equation}
\| Df(z) \|_{\MO_f}=\frac{\rho_{\widetilde{\MO}_f}(z)}{\rho_{\MO_f}(z)}=\frac{\rho_{\MO_f}(f(z))\,|f'(z)|}{\rho_{\MO_f}(z)}\geqslant \lambda>1.
\end{equation}
\end{lem}

This result shows that $f$ uniformly expands the orbifold metric at the points in $\widetilde{W}$ which are uniformly away from Siegel disks.

\subsection{The boundaries of Fatou components}

By Lemma \ref{lem:Fatou-finite}, every Fatou component of $f$ in Theorem \ref{thm:entire} is simply connected (see also \cite[Proposition 3]{EL92}).
In this subsection we will show that they are actually bounded Jordan disks. For this, we first need to prove that all immediate attracting basins are bounded.
However, the appearance of Siegel disks may cause obstruction in the proof. We use the following result to separate the Siegel disks and attracting basins.

\begin{thm}[Landing at periodic points, {\cite[Theorem 1.1]{BR20}}]\label{thm:Benini-Rempe}
Let $f\in\MB$ be a transcendental entire function with a repelling or parabolic periodic point $\zeta$.
Then there is a connected and unbounded set $E$ in the escaping set of $f$ and a period $p$ with the following properties.
\begin{enumerate}
\item $\overline{E}=E \cup\{\zeta\}$ and $\overline{E}$ does not separate the plane;
\item $f^{\circ p}(E)=E$ and $f^{\circ j}(E) \cap E=\emptyset$ for all $1 \leqslant j<p$;
\item for every $\varepsilon>0$, $f^{\circ n}$ tends to $\infty$ uniformly on $\{z \in E:|z-\zeta| \geqslant \varepsilon\}$.
\end{enumerate}
If $\widetilde{\zeta} \neq \zeta$ is a different repelling or parabolic periodic point and $\widetilde{E}$ is a set as above for $\widetilde{\zeta}$ then $E \cap \widetilde{E}=\emptyset$.
\end{thm}

In fact, Benini and Rempe proved that if further $f$ is \textit{criniferous} (in particular, if $f$ has finite order of growth), then every repelling or parabolic periodic point of $f$ is the landing
point of at least one and at most finitely many periodic hairs \cite[Theorem 1.4]{BR20}.
For partial results about the landing at periodic points of transcendental entire functions, see \cite{Mih10} and \cite{BL14}.

\begin{lem}\label{lem:bd-Fatou}
Let $f$ be a transcendental entire function in Theorem \ref{thm:entire}. Then every Fatou component of $f$ is a bounded Jordan disk.
\end{lem}

\begin{proof}
We divide the proof into several steps:

\medskip
\textit{Step 1. Reduction}.
Let $U$ be a Fatou component of $f$. If $U$ is bounded Jordan disk, then any component of $f^{-1}(U)$ is also by Proposition \ref{prop:take-pre} since each Fatou component contains at most finitely many critical points. Since we have assumed that all Siegel disks are bounded Jordan disks, by Lemma \ref{lem:Fatou-finite}, it is sufficient to prove that every immediate attracting basin is a bounded Jordan disk. As before, without loss of generality, we assume that all Siegel disks and immediate attracting basins of $f$ have period one and $f$ is postsingularly finite in the Fatou set.

\medskip
\textit{Step 2. Siegel disks and bubble structure}.
Let $\Delta:=\Delta_i$ with $1\leqslant i\leqslant p$ be a fixed Siegel disk of $f$. For any component $U_1$ of $f^{-1}(\Delta)$, the map $f:U_1\to\Delta$ is proper and $f:\partial U_1\to\partial\Delta$ is a finite covering map.
Therefore, there exists a pullback sequence $\{U_n\}_{n\geqslant 0}$ of $U_0:=\Delta$ such that
\begin{itemize}
\item each $U_n$ is bounded Jordan disk for $n\geqslant 1$;
\item $U_1\neq\Delta$ and $\overline{U}_1\cap\overline{\Delta}=\{z_0\}$ is a critical point on $\partial\Delta$; and
\item $\overline{U}_{n+1}\cap\overline{U}_n=\{z_n\}$ is an $n$-th preimage of $z_0$ for all $n\geqslant 1$.
\end{itemize}
Moreover, there exists another sequence $\{U_n'\}_{n\geqslant 0}$ with $U_0'=\Delta$ such that
\begin{itemize}
\item each $U_n'$ is bounded Jordan disk for $n\geqslant 1$;
\item $U_1'\neq\Delta$ and $\overline{U_1'}\cap\overline{\Delta}=\{z_0'\}$ is the unique preimage of $z_0$ on $\partial\Delta$; and
\item  $f(U_n')=U_n$ and $\overline{U_{n+1}'}\cap\overline{U_n'}=\{z_n'\}$ is a preimage of $z_n$ for all $n\geqslant 1$.
\end{itemize}
Then $\big(\bigcup_{n\geqslant 0} \overline{U}_n\big)\cap\big(\bigcup_{n\geqslant 0} \overline{U_n'}\big)=\emptyset$.
Since $f$ has no asymptotic values in $\C$, there are following two cases:
\begin{itemize}
\item Case (1). $\bigcup_{n\geqslant 0} \overline{U}_n$ and $\bigcup_{n\geqslant 0} \overline{U_n'}$ are bounded;
\item Case (2). $\bigcup_{n\geqslant 0} \overline{U}_n$ and $\bigcup_{n\geqslant 0} \overline{U_n'}$ are unbounded.
\end{itemize}

\medskip
\textit{Step 3. The existence of repelling fixed point}.
Since $\MP(f)\cap J(f)$ is finite outside the union of all Siegel disk boundaries, there exists $n_0\geqslant 1$ such that $\overline{U}_n\cap \MP(f)=\emptyset$ for all $n\geqslant n_0$. By Lemma \ref{lem:semi-hyperbolic}, we have $\diam_{\EC}(\overline{U}_n)\to 0$ as $n\to\infty$. In particular, we have $\diam_{\EC}(\overline{U}_n\cup \overline{U}_{n-1})\to 0$ as $n\to\infty$.
For the above Case (1), let $w_{n_k} \in \overline{U}_{n_k}$ be a sequence of points converging to a point $\zeta\in\C$. By the continuity, $f(w_{n_k}) \rightarrow f(\zeta)$ as $k\to\infty$. On the other hand, $\dist_{\C}(w_{n_k}, f(w_{n_k})) \leqslant\diam_{\C}(\overline{U}_{n_k}\cup \overline{U}_{n_k-1})\rightarrow 0$ as $k \rightarrow \infty$. So $f(\zeta)=\zeta\in J(f)$. This implies that any limit point of the sequence of sets $\{\overline{U}_n\}_{n\geqslant 0}$ is a fixed point of $f$. Since the set of all limit points is connected and the fixed points of $f$ is discrete, it follows that $\overline{U}_n$ converges to a fixed point $\zeta\in J(f)$ as $n\to\infty$. By Lemma \ref{lem:Fatou-finite}, $\zeta$ is a repelling fixed point. Moreover, $\overline{U_n'}$ converges to $\zeta'\in J(f)$ as $n\to\infty$, which is preimage of $\zeta$.

\medskip
\textit{Step 4. Separating attracting and Siegel}.
By Theorem \ref{thm:Benini-Rempe}, there is a connected and unbounded set $E$ in the escaping set of $f$ such that $\overline{E}=E\cup\{\zeta\}$. This implies that $\overline{E}\cap\partial\Delta=\emptyset$. Let $E'$ be the component of $f^{-1}(E)$ such that $\overline{E'}=E'\cup\{\zeta'\}$. Then $E'\neq E$ is a connected and unbounded set in the escaping set of $f$ satisfying $\overline{E'}\cap\partial\Delta=\emptyset$.
Let $A:=A_j$ be an immediate attracting basin of $f$ with $1\leqslant j\leqslant q$. Then $A$ is contained in a component of $\C\setminus X$, where
\begin{equation}
X:=\overline{\Delta}\cup X_1 \text{\quad and\quad} X_1:=\bigcup_{n\geqslant 1} \overline{U}_n\cup\bigcup_{n\geqslant 1} \overline{U_n'}\cup \overline{E}\cup \overline{E'}.
\end{equation}
Since $X_1\cap\partial\Delta=\emptyset$, it follows that $\C\setminus X$ has no component whose boundary contains the whole $\partial\Delta$.
Assume $\partial A\cap\partial\Delta\neq\emptyset$. Then $\partial\Delta\subset \partial A$ since $f(\partial A)=\partial A$ and the forward orbit of any point in $\partial\Delta$ under $f$ is dense in $\partial\Delta$. This is a contradiction.
Therefore, we have $\overline{\Delta}\cap\overline{A}=\emptyset$ and hence
\begin{equation}\label{equ:Sie-attr-disjoint}
\overline{\Delta}_i\cap\overline{A}_j=\emptyset \text{\quad for all } 1\leqslant i\leqslant p \text{ and } 1\leqslant j\leqslant q.
\end{equation}
For above Case (2), we still have \eqref{equ:Sie-attr-disjoint} by the same argument if we replace the above $X_1$ by $\bigcup_{n\geqslant 1} \overline{U}_n\cup\bigcup_{n\geqslant 1} \overline{U_n'}$.

\medskip
\textit{Step 5. Uniform expansion and the conclusion}.
For $1\leqslant j\leqslant q$, let $\partial A_j$ be contained a component $\widetilde{W}$ of $f^{-1}(W)$, where $W$ is defined in \eqref{equ:W}.
By \eqref{equ:Sie-attr-disjoint}, there exists $\delta>0$ such that $\overline{A}_j\cap f^{-1}(\overline{A}_j\setminus \overline{B}_j)\subset \widetilde{W}_\delta$, where $\widetilde{W}_\delta$ is defined in \eqref{equ:W-tilde-delta}.
Let $\MO_f:=(W,\nu)$ and $\widetilde{\MO}_f:=(\widetilde{W},\widetilde{\nu})$ be the hyperbolic orbifolds defined in the previous subsection.
By Lemma \ref{lem:expand-orbifold}, $f$ is uniformly expanding near the boundary of $A_j$ with respect the orbifold metric $\rho_{\MO_f}(z)|dz|$.
By a completely similar argument to \cite[Proof of Theorem 6.4, $(7)\Rightarrow(1)$]{Par22} (see also \cite[Proof of Theorem 1.10, $(h)\Rightarrow(a)$]{BFR15} or \cite[Theorem 4.1, p.\,94]{CG93}),
we conclude that $A_j$ is bounded and $\partial A_j$ is locally connected. By the maximum principle, $A_j$ is a Jordan disk.
\end{proof}

Note that we have made the assumption that $f$ is postsingularly finite outside the closure of all Siegel disks. One may use the same argument as \cite[Proposition 5.2 and Lemma 5.5]{WYZZ25} to obtain Lemma \ref{lem:bd-Fatou}.

Moreover, one can obtain that every Fatou component of $f$ is a bounded quasi-disk. By the assumption in Theorem \ref{thm:entire}, it sufficient to consider each immediate attracting basin $A$.
For each $z\in\partial A$, there exists a small neighborhood $\D(z,r)$ such that the degree of $f^{\circ n}:V_n\to \D(z,r)$ has a uniform upper bound for any component $V_n$ of $f^{-n}(\D(z,r))$. Hence $f$ is \textit{semi-hyperbolic} (i.e., $f$ satisfies the bounded degree condition) on $\partial A$.
This implies that $\partial A$ is a quasi-circle (see \cite[Theorem 1.1]{CJY94} and \cite[Proposition 6.1]{QWY12}).

\subsection{Proof of Theorem \ref{thm:entire}}\label{subsec-transcend-case}

To prove the local connectivity of a Julia set in $\EC$, we use the following criterion (see {\cite[Theorem 4.4, p.\,113]{Why42}}).

\begin{lem}[LC criterion]\label{lem:LC-criterion}
A compact subset $X$ in $\EC$ is locally connected if and only if the following two conditions hold:
\begin{enumerate}
\item  The boundary of every component of $\EC\setminus X$ is locally connected; and
\item  For any given $\varepsilon>0$, there are only finitely many components of $\EC\setminus X$ whose spherical diameters are greater than $\varepsilon$.
\end{enumerate}
\end{lem}

The Julia set of any transcendental entire function is an unbounded closed subset of $\C$. It is known that a continuum cannot fail to be locally connected only at a single point (see \cite[Corollary 5.13, p.\,78]{Nad92}). Therefore, the Julia set $J(f)$ of a transcendental entire function $f$ is locally connected in $\C$ if and only if $J(f)\cup\{\infty\}$ is locally connected in $\EC$.

\begin{proof}[Proof of Theorem \ref{thm:entire}]
By Lemma \ref{lem:bd-Fatou}, iterating $f$ if necessary, we assume that all periodic Fatou components of $f$ are fixed. Since all Fatou components of $f$ are bounded Jordan disks, by Lemma \ref{lem:LC-criterion}, it suffices to prove that for any $\varepsilon>0$, there are only finitely many Fatou components of $f$ whose spherical diameters are greater than $\varepsilon$.

Let $U_0$ be a Siegel disk or an immediate attracting basin of $f$. Since the number of critical points in every component of $F(f)$ and the local degrees of all critical points of $f$ have a uniform upper bound, there exists a number $D_0\geqslant 1$ such that $\deg(f^{\circ n}: U_n\to U_0)\leqslant D_0$ for any Fatou component $U_n$ of $f$ satisfying $f^{\circ n}(U_n)=U_0$ and $f^{\circ (n-1)}(U_n)\neq U_0$. By Lemma \ref{lem:bd-Fatou}, the map $f^{\circ n}:\partial U_n\to\partial U_0$ is a covering map with degree at most $D_0$.

If $U_0=\Delta$ is a fixed Siegel disk of $f$, by the assumptions in Theorem \ref{thm:entire}, there exist finitely many bounded Jordan disk $\{V_{0,(i)}:1\leqslant i\leqslant M_0\}$ in $\C\setminus\overline{\Delta}$ such that $\overline{V}_{0,(i)}\cap \MP(f)\neq\emptyset$ is a subarc on $\partial\Delta$ for all $1\leqslant i\leqslant M_0$ and $\partial\Delta\subset\bigcup_{i=1}^{M_0} \overline{V}_{0,(i)}$.
Let $\varepsilon>0$ be given. By Theorem \ref{thm:main-pre}, there exists $N_0\geqslant 1$ such that for all $n\geqslant N_0$ and $1\leqslant i\leqslant M_0$, we have $\diam_{\EC}(V_{n,(i)})<\varepsilon/(D_0M_0)$, where $V_{n,(i)}$ is any component of $f^{-n}(V_{0,(i)})$. Hence we have $\diam_{\EC}(U_n)<\varepsilon$ for all $n\geqslant N_0$.

If $U_0=A$ is a fixed immediate attracting basin of $f$, by the assumptions in Theorem \ref{thm:entire}, there exist finitely many bounded Jordan disk $\{W_{0,(j)}:1\leqslant j\leqslant M_1\}$ such that $W_{0,(j)}\cap \partial A\neq\emptyset$ and $\sharp(\overline{W}_{0,(j)}\cap \MP(f))\leqslant 1$ for all $1\leqslant j\leqslant M_1$ and $\partial A\subset\bigcup_{j=1}^{M_1} \overline{W}_{0,(j)}$.
By Lemma \ref{lem:semi-hyperbolic}, there exists $N_1\geqslant 1$ such that for all $n\geqslant N_1$ and $1\leqslant j\leqslant M_1$, we have $\diam_{\EC}(W_{n,(j)})<\varepsilon/(D_0M_1)$, where $W_{n,(j)}$ is any component of $f^{-n}(W_{0,(j)})$. Hence we have $\diam_{\EC}(U_n)<\varepsilon$ for all $n\geqslant N_1$.

Note that there exists a large number $R=R(\varepsilon)>1$ such that for any subset $X$ of $\EC$ which is disjoint with $\D(0,R)$, then $\diam_{\EC}(X)<\varepsilon$.
By the open mapping theorem, for any given $n\geqslant 1$, the Fatou components $U_n$'s of $f$ satisfying $f^{\circ n}(U_n)=U_0$ and $f^{\circ (n-1)}(U_n)\neq U_0$ can only accumulate at infinity. Hence for given $n\geqslant 1$, there are only finitely many of $U_n$'s have spherical diameters which are greater than $\varepsilon$. By Lemma \ref{lem:Fatou-finite}, the choice of $U_0$ is finite. The proof is complete.
\end{proof}

Note that the maps in Theorem \ref{thm:entire} cannot have parabolic basins. Based on \cite{ARS22}, we believe that one may weaken the conditions of $f$ in Theorem \ref{thm:entire} to be geometrically finite outside the closure of the Siegel disks so that the Julia sets are still locally connected (see \cite[Theorem B]{FY23}).

\section{Locally connected Julia sets with asymptotic values}\label{sec:example}

We have mentioned in the introduction that when the Fatou set of a transcendental entire function contains an asymptotic value, then this function has unbounded Fatou components and the Julia set is not locally connected. However, there exist transcendental entire functions with non-empty Fatou set and locally connected Julia sets containing asymptotic values (see \cite{Mor99} for subhyperbolic examples).
In this section, based on Theorem \ref{thm:main-pre}, we construct some transcendental entire functions with Siegel disks and locally connected Julia sets containing asymptotic values.

\subsection{A special family}

To apply Theorem \ref{thm:main-pre}, we first need to obtain some bounded type Siegel disks with quasi-circle boundaries from transcendental entire functions with asymptotic values.
To this end, we shall apply the polynomial-like renormalization theory to a family of entire functions.

Let $U$ and $V$ be two bounded Jordan disks in $\C$ such that $U\Subset V$. The map $f:U\to V$ is called a \textit{polynomial-like mapping} of degree $d\geqslant 2$ if $f$ is a proper holomorphic map of degree $d$. The \emph{filled Julia set} $K(f)$ is defined as $K(f):=\bigcap_{n\geqslant 0}f^{-n}(V)$ and the \textit{Julia set} $J(f)$ is the topological boundary of $K(f)$.
Two polynomial-like mappings $f_1:U_1\to V_1$ and $f_2:U_2\to V_2$ are said to be \emph{hybrid equivalent} if there exists a quasiconformal mapping $h$ defined from a neighborhood of $K(f_1)$ onto a neighborhood of $K(f_2)$, which conjugates $f_1$ to $f_2$ and the complex dilatation of $h$ on $K(f_1)$ is zero.

We consider the following family
\begin{equation}\label{equ:f-a}
f_\lambda(z)=\lambda ze^z, \text{\quad where } \lambda\in\C\setminus\{0\}.
\end{equation}
Note that $f_\lambda$ has exactly one critical point $-1$ and two singular values in $\C$: one is the critical value $-\lambda/e$ and the other is the asymptotic value $0$ which is also a fixed point.
In particular, if $|\lambda|>1$, then $0$ is a repelling fixed point.
The family $f_\lambda$ has been studied in \cite{Fag95}, \cite{Kre01} and \cite{Gey01} etc.
Based on Douady-Hubbard's straightening theorem (\cite{DH85b}) and the combinatorial structure of fundamental domains in the dynamical planes of $f_\lambda$,
Fagella proved the following result (see \cite[\S 4.3]{Fag95} and Figure \ref{Fig:parameter-exp}).

\begin{lem}\label{lem:Fagella}
There exist infinitely homeomorphic copies $\mathcal{M}_k$, where $k=\pm 1,\pm 2,\cdots$, of the Mandelbrot set $M$ in the parameter space $\C\setminus\overline{\D}$ of $f_\lambda$ such that for each $c\in M$, there exist $\lambda=\lambda(c)\in\mathcal{M}_k$ and two bounded Jordan disks $U_\lambda$, $V_\lambda$ containing the critical point $-1$ such that $f_\lambda^{\circ 2}:U_\lambda\to V_\lambda$ is a quadratic-like mapping and hybird equivalent to $z\mapsto z^2+c$.
\end{lem}

\begin{figure}[htbp]
  \setlength{\unitlength}{1mm}
  \includegraphics[width=0.47\textwidth]{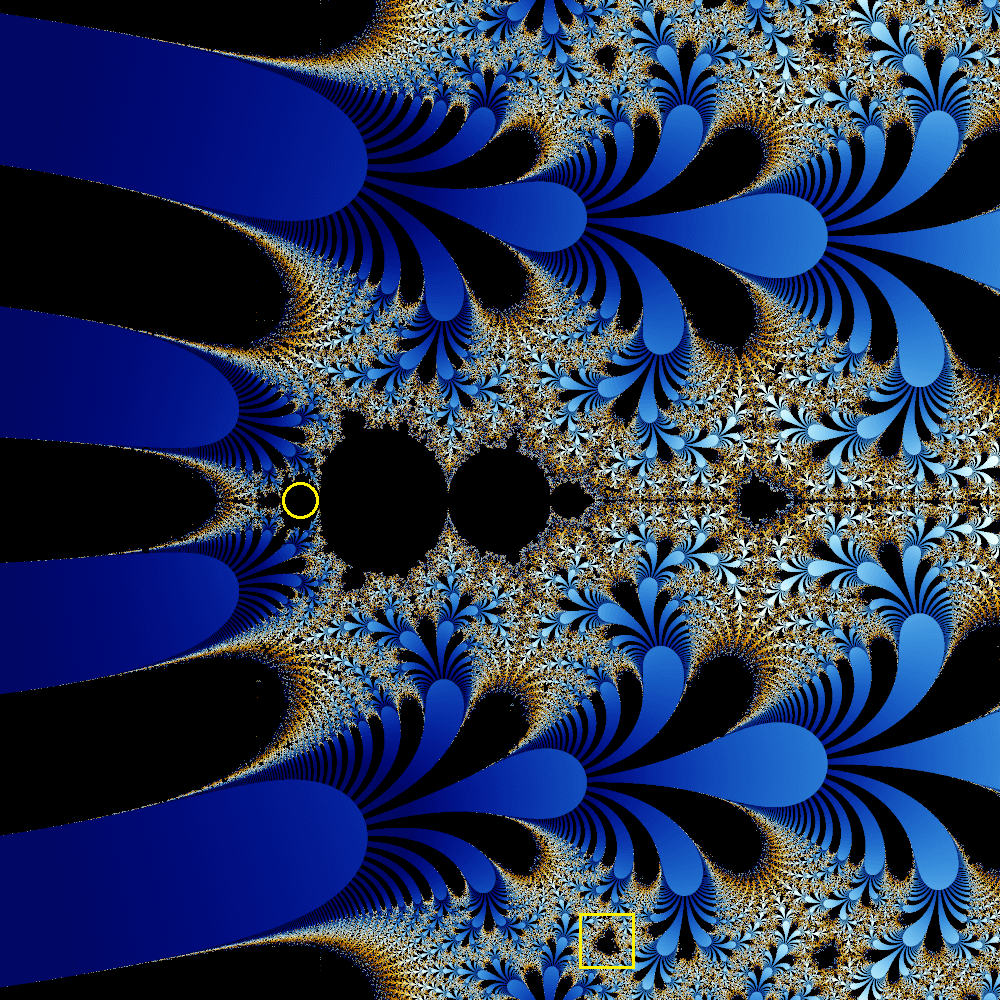}\quad
  \includegraphics[width=0.47\textwidth]{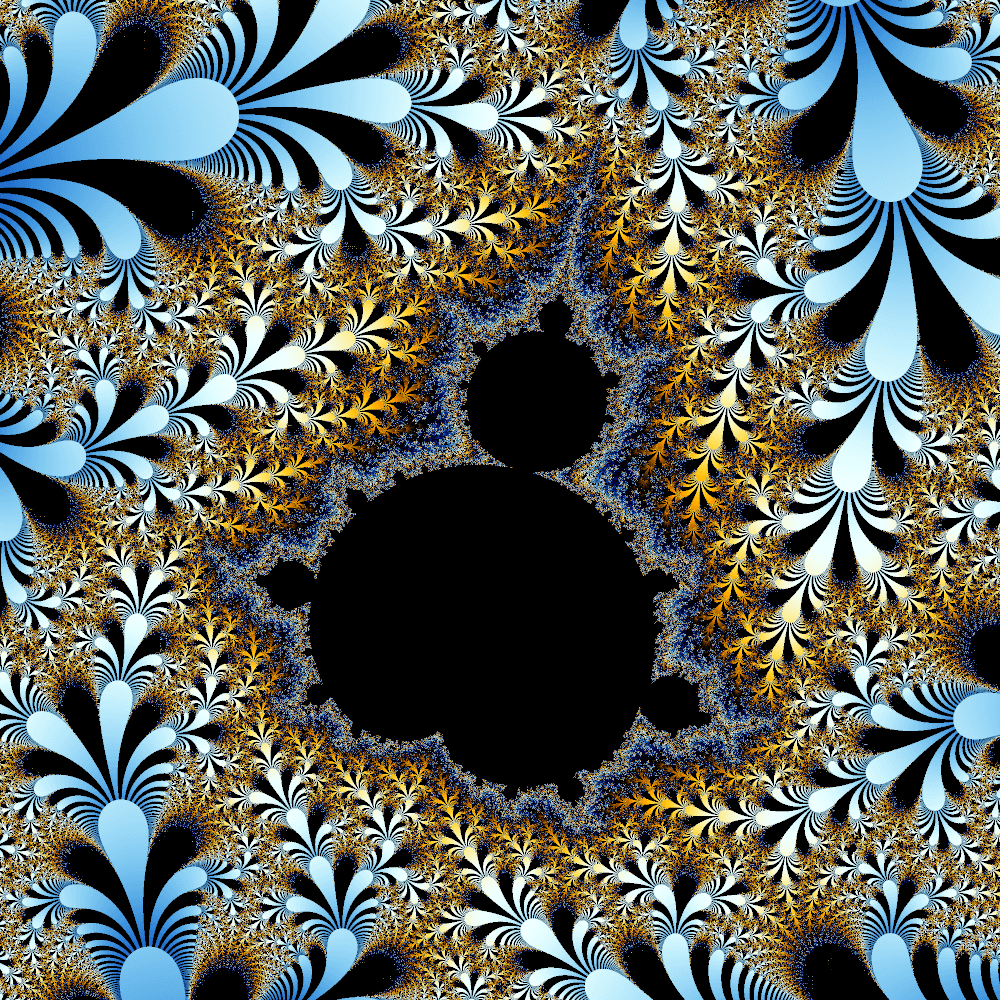}
  \caption{Left: The parameter plane of $f_\lambda(z)=\lambda z e^z$, where the unit circle is marked by a yellow circle. Right: A zoom of the left picture near a Mandelbrot copy, which is inclosed by a yellow box. See also \cite[Figures 8 and 28]{Fag95}.}
  \label{Fig:parameter-exp}
\end{figure}

Based on \cite{Dou87} and \cite{Her87}, as an immediate consequence of Lemma \ref{lem:Fagella}, we have the following result.

\begin{cor}\label{cor:Siegel-polylike}
For each integer $k\neq 0$ and bounded type irrational number $\theta$, there exists $\lambda\in \mathcal{M}_k\subset\C\setminus\overline{\D}$ and two bounded Jordan disks $U_\lambda$, $V_\lambda$ such that $f_\lambda^{\circ 2}:U_\lambda\to V_\lambda$ is hybird equivalent to $z\mapsto e^{2\pi\ii\theta}z+z^2$. In particular, $f_\lambda$ has a $2$-cycle of bounded type Siegel disks whose boundary components are quasi-circles in $\C$ and one of the components contains the critical point $-1$.
\end{cor}

\subsection{Proof of Theorem \ref{thm:lc-asymp}}

Based on Corollary \ref{cor:Siegel-polylike}, we are in the position to prove Theorem \ref{thm:lc-asymp}.

\begin{proof}[Proof of Theorem \ref{thm:lc-asymp}]
By Corollary \ref{cor:Siegel-polylike}, let $\lambda\in \mathcal{M}_k\subset\C\setminus\overline{\D}$ be given such that $f_\lambda$ has a $2$-cycle of bounded type Siegel disks $\{\Delta_1,\Delta_2\}$ such that $\partial\Delta_1$ and $\partial \Delta_2$ are quasi-circles and $-1\in\partial\Delta_1\cup\partial\Delta_2$.
Since $f_\lambda$ has exactly two singular values $-\lambda/e\in\partial\Delta_1\cup\partial\Delta_2$ and $0$ (which is a repelling fixed point) in $\C$, it follows that $f_\lambda$ has no wandering domains and Baker domains (\cite{GK86}, \cite{EL92}). Moreover, $f_\lambda$ has neither attracting basins nor parabolic basins (see \cite[\S8-\S10]{Mil06}. This implies that the Fatou set of $f_\lambda$ is the union of two periodic bounded type Siegel disks and their preimages.
According to Lemma \ref{lem:take-pre}, for each Fatou component $U$ of $f_\lambda$, the map $f_\lambda:U\to f_\lambda(U)$ is conformal and hence $U$ is a Jordan disk (actually, a quasi-circle).
Applying Theorem \ref{thm:main-pre} to $f_\lambda^{\circ 2}$ and by a completely similar argument as the proof of Theorem \ref{thm:entire}, we conclude that for any $\varepsilon>0$, there are only finitely many Fatou components of $f_\lambda$ whose spherical diameters are greater than $\varepsilon$, and the Julia set of $f_\lambda$ is locally connected by Lemma \ref{lem:LC-criterion}.

\medskip
For the chosen $\lambda$ above and $f:=f_\lambda$, the map $g(z)=e^z+z+\log\lambda$ satisfies $f(e^z)=e^{g(z)}$.
This map is not contained in the class $\MB$.
Let $\gamma:[0,+\infty)\to\C$ be a simple curve in $\C\setminus(\overline{\Delta}_1\cup\overline{\Delta}_2)$ which connects $\gamma(0)=0$ and $\infty=\lim_{t\to\infty}\gamma(t)$. Then $\{z\in\C: e^z\in\C\setminus\gamma\}$ has infinitely many components and each of them contains a $2$-cycle of Siegel disks of $g$.
Since the Julia set of $f$ is locally connected, the Julia set of $g$ is also.
\end{proof}

\begin{rmk}
Note that the Siegel disks in Theorem \ref{thm:lc-asymp} have period two. An interesting problem is to find a transcendental entire function having a \textit{fixed} Siegel disk and a locally connected Julia set with asymptotic values. Consider $f_\rho(z)=e^{1-\rho}z e^z$, where $\rho=e^{2\pi\ii\theta}$ and $\theta$ is a bounded type irrational number. The computer experiment shows that $f_\rho$ has a Siegel disk centered at $\rho-1$ with rotation number $\theta$ whose boundary is a quasi-circle containing the critical point $-1$.
\end{rmk}

\bibliographystyle{amsalpha}
\bibliography{E:/Latex-model/Ref1}

\newcommand{\etalchar}[1]{$^{#1}$}
\providecommand{\bysame}{\leavevmode\hbox to3em{\hrulefill}\thinspace}
\providecommand{\MR}{\relax\ifhmode\unskip\space\fi MR }
\providecommand{\MRhref}[2]{%
  \href{http://www.ams.org/mathscinet-getitem?mr=#1}{#2}
}
\providecommand{\href}[2]{#2}
\begin{thebibliography}{BHK{\etalchar{+}}93}

\bibitem[ABC04]{ABC04}
A.~Avila, X.~Buff, and A.~Ch{\'{e}}ritat, \emph{Siegel disks with smooth
  boundaries}, Acta Math. \textbf{193} (2004), no.~1, 1--30.

\bibitem[ARS22]{ARS22}
M.~Alhamed, L.~Rempe, and D.~Sixsmith, \emph{Geometrically finite
  transcendental entire functions}, J. Lond. Math. Soc. (2) \textbf{106}
  (2022), no.~2, 485--527.

\bibitem[Bak68]{Bak68}
I.~N. Baker, \emph{Repulsive fixpoints of entire functions}, Math. Z.
  \textbf{104} (1968), 252--256.

\bibitem[BC07]{BC07}
X.~Buff and A.~Ch{\'{e}}ritat, \emph{How regular can the boundary of a
  quadratic {S}iegel disk be?}, Proc. Amer. Math. Soc. \textbf{135} (2007),
  no.~4, 1073--1080.

\bibitem[BD00]{BD00b}
I.~N. Baker and P.~Dom{\'{\i}}nguez, \emph{Some connectedness properties of
  {J}ulia sets}, Complex Variables Theory Appl. \textbf{41} (2000), no.~4,
  371--389.

\bibitem[BF10]{BF10}
R.~Berenguel and N.~Fagella, \emph{An entire transcendental family with a
  persistent {S}iegel disc}, J. Difference Equ. Appl. \textbf{16} (2010),
  no.~5-6, 523--553.

\bibitem[BF14]{BF14a}
B.~Branner and N.~Fagella, \emph{Quasiconformal surgery in holomorphic
  dynamics}, Cambridge Studies in Advanced Mathematics, vol. 141, Cambridge
  University Press, Cambridge, 2014.

\bibitem[BF18]{BF18}
A.~M. Benini and N.~Fagella, \emph{Singular values and bounded {S}iegel disks},
  Math. Proc. Cambridge Philos. Soc. \textbf{165} (2018), no.~2, 249--265.

\bibitem[BFJK25]{BFJK25}
K.~Bara\'{n}ski, N.~Fagella, X.~Jarque, and B.~Karpi\'{n}ska, \emph{Local
  connectivity of boundaries of tame {F}atou components of meromorphic
  functions}, Math. Ann. \textbf{391} (2025), no.~2, 1779--1843.

\bibitem[BFR15]{BFR15}
W.~Bergweiler, N.~Fagella, and L.~Rempe, \emph{Hyperbolic entire functions with
  bounded {F}atou components}, Comment. Math. Helv. \textbf{90} (2015), no.~4,
  799--829.

\bibitem[BHK{\etalchar{+}}93]{BHKMT93}
W.~Bergweiler, M.~Haruta, H.~Kriete, H.~Meier, and N.~Terglane, \emph{On the
  limit functions of iterates in wandering domains}, Ann. Acad. Sci. Fenn. Ser.
  A I Math. \textbf{18} (1993), no.~2, 369--375.

\bibitem[BL14]{BL14}
A.~M. Benini and M.~Lyubich, \emph{Repelling periodic points and landing of
  rays for post-singularly bounded exponential maps}, Ann. Inst. Fourier
  (Grenoble) \textbf{64} (2014), no.~4, 1493--1520.

\bibitem[BM02]{BM02}
W.~Bergweiler and S.~Morosawa, \emph{Semihyperbolic entire functions},
  Nonlinearity \textbf{15} (2002), no.~5, 1673--1684.

\bibitem[BR20]{BR20}
A.~M. Benini and L.~Rempe, \emph{A landing theorem for entire functions with
  bounded post-singular sets}, Geom. Funct. Anal. \textbf{30} (2020), no.~6,
  1465--1530.

\bibitem[CE18]{CE18}
A.~Ch{\'{e}}ritat and A.~L. Epstein, \emph{Bounded type {S}iegel disks of
  finite type maps with few singular values}, Sci. China Math. \textbf{61}
  (2018), no.~12, 2139--2156.

\bibitem[CG93]{CG93}
L.~Carleson and T.~W. Gamelin, \emph{Complex dynamics}, Universitext: Tracts in
  Mathematics, Springer-Verlag, New York, 1993.

\bibitem[Ch{\'{e}}06]{Che06}
A.~Ch{\'{e}}ritat, \emph{Ghys-like models for {L}avaurs and simple entire
  maps}, Conform. Geom. Dyn. \textbf{10} (2006), 227--256.

\bibitem[Che22]{Che22b}
D.~Cheraghi, \emph{Topology of irrationally indifferent attractors}, arXiv:
  1706.02678v3, 2022.

\bibitem[CJY94]{CJY94}
L.~Carleson, P.~W. Jones, and J.-C. Yoccoz, \emph{Julia and {J}ohn}, Bol. Soc.
  Brasil. Mat. (N.S.) \textbf{25} (1994), no.~1, 1--30.

\bibitem[DH85]{DH85b}
A.~Douady and J.~H. Hubbard, \emph{On the dynamics of polynomial-like
  mappings}, Ann. Sci. \'{E}cole Norm. Sup. (4) \textbf{18} (1985), no.~2,
  287--343.

\bibitem[Dou87]{Dou87}
A.~Douady, \emph{Disques de {S}iegel et anneaux de {H}erman}, Bourbaki seminar,
  {V}ol. 1986/87, Ast\'{e}risque, no. 152-153, Soc. Math. France, Paris, 1987,
  pp.~151--172.

\bibitem[EL92]{EL92}
A.~E. Eremenko and M.~Lyubich, \emph{Dynamical properties of some classes of
  entire functions}, Ann. Inst. Fourier (Grenoble) \textbf{42} (1992), no.~4,
  989--1020.

\bibitem[Ere89]{Ere89}
A.~E. Eremenko, \emph{On the iteration of entire functions}, Dynamical systems
  and ergodic theory ({W}arsaw, 1986), Banach Center Publ., vol.~23, PWN,
  Warsaw, 1989, pp.~339--345.

\bibitem[Fag95]{Fag95}
N.~Fagella, \emph{Limiting dynamics for the complex standard family}, Internat.
  J. Bifur. Chaos Appl. Sci. Engrg. \textbf{5} (1995), no.~3, 673--699.

\bibitem[FY23]{FY23}
Y.~Fu and F.~Yang, \emph{Mating {S}iegel and parabolic quadratic polynomials},
  Nonlinearity, to appear, arXiv: 2305.15180, 2023.

\bibitem[Gey01]{Gey01}
L.~Geyer, \emph{Siegel discs, {H}erman rings and the {A}rnold family}, Trans.
  Amer. Math. Soc. \textbf{353} (2001), no.~9, 3661--3683.

\bibitem[GK86]{GK86}
L.~R. Goldberg and L.~Keen, \emph{A finiteness theorem for a dynamical class of
  entire functions}, Ergodic Theory Dynam. Systems \textbf{6} (1986), no.~2,
  183--192.

\bibitem[G{\'{S}}03]{GS03}
J.~Graczyk and G.~{\'{S}}wi{\c{a}}tek, \emph{Siegel disks with critical points
  in their boundaries}, Duke Math. J. \textbf{119} (2003), no.~1, 189--196.

\bibitem[Her85]{Her85}
M.~R. Herman, \emph{Are there critical points on the boundaries of singular
  domains?}, Comm. Math. Phys. \textbf{99} (1985), no.~4, 593--612.

\bibitem[Her87]{Her87}
\bysame, \emph{Conjugaison quasi-symm\'{e}trique des hom\'{e}omorphismes
  analytiques du cercle \`{a} des rotations},
  \href{https://www.math.univ-toulouse.fr/~cheritat/Herman/Herman_man_1_eng.pdf}{preliminary
  manuscript}, 1987.

\bibitem[HY98]{HY98}
X.~H. Hua and C.~C. Yang, \emph{Dynamics of transcendental functions}, Asian
  Mathematics Series, vol.~1, Gordon and Breach Science Publishers, Amsterdam,
  1998.

\bibitem[Kat11]{Kat11}
K.~Katagata, \emph{Dynamics of transcendental entire functions with {S}iegel
  disks and its applications}, Bull. Korean Math. Soc. \textbf{48} (2011),
  no.~4, 713--724.

\bibitem[KN22]{KN22b}
M.~Kisaka and H.~Naba, \emph{Some transcendental entire functions with
  irrationally indifferent fixed points}, Kodai Math. J. \textbf{45} (2022),
  no.~3, 369--387.

\bibitem[Kre01]{Kre01}
D.~Kremer, \emph{Some investigations on the dynamics of the family {$\lambda
  z\exp(z)$}}, Complex Variables Theory Appl. \textbf{45} (2001), no.~4,
  355--370.

\bibitem[KZ09]{KZ09}
L.~Keen and G.~Zhang, \emph{Bounded-type {S}iegel disks of a one-dimensional
  family of entire functions}, Ergodic Theory Dynam. Systems \textbf{29}
  (2009), no.~1, 137--164.

\bibitem[LM97]{LM97}
M.~Lyubich and Y.~Minsky, \emph{Laminations in holomorphic dynamics}, J.
  Differential Geom. \textbf{47} (1997), no.~1, 17--94.

\bibitem[Ma{\~{n}}93]{Man93}
R.~Ma{\~{n}}{\'{e}}, \emph{On a theorem of {F}atou}, Bol. Soc. Brasil. Mat.
  (N.S.) \textbf{24} (1993), no.~1, 1--11.

\bibitem[McM94]{McM94b}
C.~T. McMullen, \emph{Complex dynamics and renormalization}, Annals of
  Mathematics Studies, vol. 135, Princeton University Press, Princeton, NJ,
  1994.

\bibitem[Mih10]{Mih10}
H.~Mihaljevi{\'{c}}{-Brandt}, \emph{A landing theorem for dynamic rays of
  geometrically finite entire functions}, J. Lond. Math. Soc. (2) \textbf{81}
  (2010), no.~3, 696--714.

\bibitem[Mih12]{Mih12}
\bysame, \emph{Semiconjugacies, pinched {C}antor bouquets and hyperbolic
  orbifolds}, Trans. Amer. Math. Soc. \textbf{364} (2012), no.~8, 4053--4083.

\bibitem[Mil06]{Mil06}
J.~Milnor, \emph{Dynamics in one complex variable}, third ed., Annals of
  Mathematics Studies, vol. 160, Princeton University Press, Princeton, NJ,
  2006.

\bibitem[Mor99]{Mor99}
S.~Morosawa, \emph{Local connectedness of {J}ulia sets for transcendental
  entire functions}, Nonlinear analysis and convex analysis ({N}iigata, 1998),
  World Sci. Publ., River Edge, NJ, 1999, pp.~266--273.

\bibitem[MS93]{MS93}
W.~de Melo and S.~van Strien, \emph{One-dimensional dynamics}, Ergebnisse der
  Mathematik und ihrer Grenzgebiete (3), vol.~25, Springer-Verlag, Berlin,
  1993.

\bibitem[Nad92]{Nad92}
S.~B. Nadler, Jr., \emph{Continuum theory, \textup{An introduction}},
  Monographs and Textbooks in Pure and Applied Mathematics, vol. 158, Marcel
  Dekker, Inc., New York, 1992.

\bibitem[Osb13]{Osb13}
J.~W. Osborne, \emph{Spiders' webs and locally connected {J}ulia sets of
  transcendental entire functions}, Ergodic Theory Dynam. Systems \textbf{33}
  (2013), no.~4, 1146--1161.

\bibitem[Par22]{Par22}
L.~Pardo{-Sim\'{o}n}, \emph{Orbifold expansion and entire functions with
  bounded {F}atou components}, Ergodic Theory Dynam. Systems \textbf{42}
  (2022), no.~5, 1807--1846.

\bibitem[Pet96]{Pet96}
C.~L. Petersen, \emph{Local connectivity of some {J}ulia sets containing a
  circle with an irrational rotation}, Acta Math. \textbf{177} (1996), no.~2,
  163--224.

\bibitem[Pet98]{Pet98}
\bysame, \emph{Puzzles and {S}iegel disks}, Progress in holomorphic dynamics,
  Pitman Res. Notes Math. Ser., vol. 387, Longman, Harlow, 1998, pp.~50--85.

\bibitem[PZ04]{PZ04}
C.~L. Petersen and S.~Zakeri, \emph{On the {J}ulia set of a typical quadratic
  polynomial with a {S}iegel disk}, Ann. of Math. (2) \textbf{159} (2004),
  no.~1, 1--52.

\bibitem[QWY12]{QWY12}
W.~Qiu, X.~Wang, and Y.~Yin, \emph{Dynamics of {M}c{M}ullen maps}, Adv. Math.
  \textbf{229} (2012), no.~4, 2525--2577.

\bibitem[Rem04]{Rem04}
L.~Rempe, \emph{On a question of {H}erman, {B}aker and {R}ippon concerning
  {S}iegel disks}, Bull. London Math. Soc. \textbf{36} (2004), no.~4, 516--518.

\bibitem[Rem08]{Rem08}
\bysame, \emph{Siegel disks and periodic rays of entire functions}, J. Reine
  Angew. Math. \textbf{624} (2008), 81--102.

\bibitem[Rem09]{Rem09b}
\bysame, \emph{Rigidity of escaping dynamics for transcendental entire
  functions}, Acta Math. \textbf{203} (2009), no.~2, 235--267.

\bibitem[RS11]{RS11a}
L.~Rempe and S.~van Strien, \emph{Absence of line fields and {M}a\~{n}\'{e}'s
  theorem for nonrecurrent transcendental functions}, Trans. Amer. Math. Soc.
  \textbf{363} (2011), no.~1, 203--228.

\bibitem[She18]{She18a}
L.~Shen, \emph{An application of the degenerate {B}eltrami equation: quadratic
  polynomials with a {S}iegel disk}, Ann. Acad. Sci. Fenn. Math. \textbf{43}
  (2018), no.~1, 267--277.

\bibitem[Shi01]{Shi01}
M.~Shishikura, \emph{Herman's theorem on quasisymmetric linearization of
  analytic circle homemorphisms}, manuscript, 2001.

\bibitem[Sie42]{Sie42}
C.~L. Siegel, \emph{Iteration of analytic functions}, Ann. of Math. (2)
  \textbf{43} (1942), 607--612.

\bibitem[SY24]{SY24}
M.~Shishikura and F.~Yang, \emph{The high type quadratic {S}iegel disks are
  {J}ordan domains}, J. Eur. Math. Soc. (JEMS) (2024), 62 pp, DOI:
  \href{https://doi.org/10.4171/JEMS/1481}{10.4171/JEMS/1481}.

\bibitem[Thu84]{Thu84}
W.~P. Thurston, \emph{On the combinatorics and dynamics of iterated rational
  maps}, preprint, 1984.

\bibitem[TY96]{TY96}
L.~Tan and Y.~Yin, \emph{Local connectivity of the {J}ulia set for
  geometrically finite rational maps}, Sci. China Ser. A \textbf{39} (1996),
  no.~1, 39--47.

\bibitem[Why42]{Why42}
G.~T. Whyburn, \emph{Analytic {T}opology}, American Mathematical Society
  Colloquium Publications, v. 28, American Mathematical Society, New York,
  1942.

\bibitem[WYZZ25]{WYZZ25}
S.~Wang, F.~Yang, G.~Zhang, and Y.~Zhang, \emph{Local connectivity of {J}ulia
  sets of some rational maps with {S}iegel disks}, J. Lond. Math. Soc., to
  appear, arXiv: 2106.07450v6, 2025.

\bibitem[Yam99]{Yam99}
M.~Yampolsky, \emph{Complex bounds for renormalization of critical circle
  maps}, Ergodic Theory Dynam. Systems \textbf{19} (1999), no.~1, 227--257.

\bibitem[Yan13]{Yan13}
F.~Yang, \emph{On the dynamics of a family of entire functions}, Acta Math.
  Sin. (Engl. Ser.) \textbf{29} (2013), no.~11, 2047--2072.

\bibitem[Yan23]{Yan23J}
J.~Yang, \emph{Local connectivity of polynomial {J}ulia sets at bounded type
  {S}iegel boundaries}, arXiv: 2010.14003v4, 2023.

\bibitem[Zak99]{Zak99}
S.~Zakeri, \emph{Dynamics of cubic {S}iegel polynomials}, Comm. Math. Phys.
  \textbf{206} (1999), no.~1, 185--233.

\bibitem[Zak10]{Zak10}
\bysame, \emph{On {S}iegel disks of a class of entire maps}, Duke Math. J.
  \textbf{152} (2010), no.~3, 481--532.

\bibitem[ZFS20]{ZFS20}
S.~Zhang, J.~Fu, and X.~Shi, \emph{Bounded type {S}iegel disks of a family of
  sine families}, J. Math. Anal. Appl. \textbf{488} (2020), no.~1, 124041, 12
  pp.

\bibitem[Zha05]{Zha05}
G.~Zhang, \emph{On the dynamics of {$e^{2\pi i\theta}\sin(z)$}}, Illinois J.
  Math. \textbf{49} (2005), no.~4, 1171--1179.

\bibitem[Zha11]{Zha11}
\bysame, \emph{All bounded type {S}iegel disks of rational maps are
  quasi-disks}, Invent. Math. \textbf{185} (2011), no.~2, 421--466.

\bibitem[Zha14]{Zha14}
\bysame, \emph{Polynomial {S}iegel disks are typically {J}ordan domains},
  arXiv: 1208.1881v3, 2014.

\bibitem[Zha16]{Zha16}
\bysame, \emph{On {PZ} type {S}iegel disks of the sine family}, Ergodic Theory
  Dynam. Systems \textbf{36} (2016), no.~3, 973--1006.

\end{thebibliography}

\end{document}